\documentclass[11pt, a4paper]{amsart}
\usepackage{amsmath, amsthm, amsfonts, amssymb, mathtools}
\usepackage{a4wide}
\usepackage[latin1]{inputenc}

\usepackage{enumitem}

\newcommand{\SL}{\mathrm{SL}}

\newcommand{\N}{\mathbb{N}}
\newcommand{\Z}{\mathbb{Z}}
\newcommand{\R}{\mathbb{R}}
\newcommand{\C}{\mathbb{C}}
\newcommand{\Q}{\mathbb{Q}}
\renewcommand{\H}{\mathbb{H}}

\DeclareMathOperator{\sgn}{sgn}
\DeclareMathOperator{\Res}{Res}
\DeclareMathOperator{\im}{Im}

\DeclareMathOperator{\Int}{Int}

\renewcommand{\Im}{\mathrm{Im}}
\renewcommand{\Re}{\mathrm{Re}}

\numberwithin{equation}{section}
	\newtheorem{theorem}{Theorem}[section]
	\newtheorem{lemma}[theorem]{Lemma}
	\newtheorem{proposition}[theorem]{Proposition} 
	\newtheorem{corollary}[theorem]{Corollary}
	\newtheorem{conjecture}[theorem]{Conjecture}
	
	\theoremstyle{definition} 
	
	\newtheorem{example}[theorem]{Example}
	\newtheorem{remark}[theorem]{Remark}
\date{\today}

\author{Steffen L\"obrich and Markus Schwagenscheidt}

\address{Korteweg-de Vries Institute for Mathematics, University of Amsterdam, Science Park 105-107, 1098 XG Amsterdam, The Netherlands}
\email{s.loebrich@uva.nl}

\address{Mathematical Institute, University of Cologne, Weyertal 86-90, D--50931 Cologne, Germany}
\email{mschwage@math.uni-koeln.de}

\title{Meromorphic modular forms with rational cycle integrals}

\thanks{The work of the first author is supported by ERC starting grant H2020 ERC StG \#640159. The second author is supported by the SFB-TRR 191 \lq Symplectic Structures in Geometry, Algebra and Dynamics\rq, funded by the DFG}

\allowdisplaybreaks

\begin{document}

\begin{abstract}
	We study rationality properties of geodesic cycle integrals of meromorphic modular forms associated to positive definite binary quadratic forms. In particular, we obtain finite rational formulas for the cycle integrals of suitable linear combinations of these meromorphic modular forms. 
\end{abstract}

\maketitle

\section{Introduction}
\label{sec:introduction}

One of the fundamental results in the classical theory of modular forms is the fact that the vector spaces of modular forms are spanned by forms with rational Fourier coefficients. Besides that, there are other natural rational structures on these spaces, for example coming from the rationality of periods or cycle integrals of modular forms. This was first shown by Kohnen and Zagier in \cite{kohnenzagierrationalperiods}, where they proved the rationality of the even periods of the cusp forms
\begin{align}\label{fkdef}
f_{k,D}(z) := \frac{|D|^{k-\frac{1}{2}}}{\pi}\sum_{Q \in \mathcal{Q}_{D}}Q(z,1)^{-k}
\end{align}
of weight $2k$ for $\Gamma(1) = \SL_{2}(\Z)$, for $k \geq 2$ and all discriminants $D > 0$. Here the sum runs over the set $\mathcal{Q}_{D}$ of all integral binary quadratic forms of discriminant $D$. These cusp forms were introduced by Zagier while investigating the Doi-Naganuma lift in \cite{zagierdoinaganuma}, and they played a prominent role in the explicit description of the Shimura-Shintani correspondence in \cite{kohnenzagiercriticalstrip}. The aforementioned rationality result of Kohnen and Zagier was generalized to Fuchsian groups of the first kind by Katok \cite{katok}. Periods and cycle integrals of other types of modular forms, such as weakly holomorphic modular forms, harmonic Maass forms, or meromorphic modular forms, have been the object of active research over the last years, see for example \cite{bringmannfrickekent, bringmannguerzhoykane, bif, bifl, dit}.

If we allow negative discriminants $D < 0$ in \eqref{fkdef} and restrict the summation to positive definite forms $Q \in \mathcal{Q}_{D}$, then we obtain meromorphic modular forms $f_{k,D}$ of weight $2k$ for $\Gamma(1)$ with poles of order $k$ at the CM points of discriminant $D$. These forms recently attracted some attention, starting with the work of Bengoechea \cite{bengoecheapaper} on the rationality properties of their Fourier coefficients. Their regularized inner products and connections to locally harmonic Maass forms were investigated by Bringmann, Kane, and von Pippich \cite{bringmannkanevonpippich} and the first author \cite{loebrich}. Furthermore, Zemel \cite{zemel} used them to prove a higher-dimensional analogue of the Gross-Kohnen-Zagier theorem \cite{grosskohnenzagier}, which hints at a deeper geometric meaning of the meromorphic $f_{k,D}$. Recently, Alfes-Neumann, Bringmann, and the second author in \cite{anbs} established modularity properties of the generating series of traces of cycle integrals of $f_{k,D}$ and used this to show the rationality of suitable linear combinations of these traces. 

It is natural to ask whether the individual cycle integrals of the meromorphic modular forms $f_{k,D}$ for $D < 0$ have nice rationality properties too. For an indefinite integral binary quadratic form $A = [a,b,c]$ of non-square discriminant the cycle integral of $f_{k,D}$ along the closed geodesic corresponding to $A$ is defined by
\begin{align*}
	\mathcal{C}(f_{k,D},A) := \int_{\Gamma(1)_{A}\backslash S_{A}}f_{k,D}(z)A(z,1)^{k-1}dz,
\end{align*}
where 
\[
S_{A} := \{z \in \H : a|z|^{2}+b\Re(z)+c = 0\}
\]
is a semi-circle centered at the real line and $\Gamma(1)_{A}$ denotes the stabilizer of $A$ in $\Gamma(1)$. Note that, due to the modularity of $f_{k,D}$, the cycle integral depends only on the $\Gamma(1)$-equivalence class of $A$. If $f_{k,D}$ has a pole on $S_A$, the cycle integral can be defined as a Cauchy principal value, see Section~\ref{section cycle integrals}. Numerical integration yields the following approximations for $k \in \{2,4,6\}$ and $D = -3$.
\begin{align*}
	 \renewcommand{\arraystretch}{1.2}
	\begin{array}{|c||c|c|c|c|c|c|}
	\hline 
	A & [1,1,-1] & [1,0,-2] & [1,1,-3] & [1,1,-4] & [1,1,-5] & [1,0,-6] \\ 
	 \hline \hline
	 \mathcal{C}(f_{2,-3},A) & 4 & 8 & 12 & 28 & 10 & 16\\
	 \hline
	\mathcal{C}(f_{4,-3},A) & 20 & 48 & 92 & 452 & 170 & 288 \\
	\hline 
	\mathcal{C}(f_{6,-3},A) & 142.36448 & 411.27103 & 1049.99067 & 12351.27103 & 5635.65417 & 8944.31786\\
	\hline
	\end{array}
\end{align*}
It seems that the cycle integrals of $f_{2,-3}$ and $f_{4,-3}$ are integers, but there is little reason to believe that the cycle integrals of $f_{6,-3}$ are rational numbers. 

The main aim of the present work is to investigate the rationality of the cycle integrals of $f_{k,D}$ for $D < 0$. As we will see, the failure of rationality of these cycle integrals is due to the existence of cusp forms of weight $2k$. In particular, we have to take certain linear combinations of cycle integrals of a fixed $f_{k,D}$ or a fixed cycle integral of linear combinations of forms $f_{k,D}$ to obtain convenient rationality results. We also treat forms of higher level $\Gamma_{0}(N)$, as well as the case $k = 1$. We remark that our results generalize the rationality results of \cite{anbs} in several aspects, using a very different proof.

\section{Statement of results}\label{section results}

Let $N$ and $k$ be positive integers and let $\Gamma = \Gamma_{0}(N)$. For any $D \in \Z$ the group $\Gamma$ acts on the set $\mathcal{Q}_{D}$ of (positive definite if $D < 0$) integral binary quadratic forms $Q = [a,b,c]$ of discriminant $D = b^{2}-4ac$ with $N \mid a$, with finitely many orbits if $D \neq 0$. We write $[Q_{0}]$ for the $\Gamma$-class of $Q_{0} \in \mathcal{Q}_{D}$. For $D \neq 0$ and $k \geq 2$ we define the associated function
\[
f_{k,Q_{0}}(z) := \frac{|D|^{k-\frac{1}{2}}}{\pi}\sum_{Q \in [Q_{0}]}Q(z,1)^{-k}
\]
on $\H$. For $k = 1$ the function $f_{1,Q_{0}}(z)$ is defined using Hecke's trick, see Section~\ref{section modular forms quadratic forms}. Throughout, we let $A \in \mathcal{Q}_{D}$ denote an indefinite quadratic form of non-square discriminant $D > 0$, and $P \in \mathcal{Q}_{d}$ a positive definite quadratic form of discriminant $d < 0$. Then $f_{k,A}$ is a cusp form of weight $2k$ for $\Gamma$, and $f_{k,P}$ is a meromorphic modular form of weight $2k$ for $\Gamma$ which has poles of order $k$ at the CM points $\tau_{Q} \in \H$ (defined by $Q(\tau_{Q},1) =0$) for $Q \in [P]$. 

Our explicit formulas for the cycle integrals of $f_{k,P}$ will be given in terms of the following function. For $k \geq 2$, an indefinite quadratic form $A \in \mathcal{Q}_{D}$ of non-square discriminant $D > 0$, and $\tau \in \H$ not lying on any of the semi-circles $S_{Q}$ for $Q \in [A]$, we define the function
\begin{align}\label{eq def local polynomial}
\begin{split}
\mathcal{P}_{k,A}(\tau) &:=D^{k-\frac12}\frac{(-1)^k\zeta_{\Gamma,A}(k)+ \zeta_{\Gamma,-A}(k)}{2^{k-2}(2k-1)\Im(\tau)^{k-1}}
 \\
 & \qquad +2 \left(-i\sqrt{D}\right)^{k-1}\sum_{\substack{Q = [a,b,c] \in [A] \\\tau \in \Int(S_{Q})}}\operatorname{sgn}(a) P_{k-1}\left(\frac{i(a|\tau|^{2}+b\Re(\tau)+c)}{\Im(\tau)\sqrt{D}}\right),
\end{split}
\end{align}
where the zeta function $\zeta_{\Gamma,A}(s)$ is defined in \eqref{zetadef}, $P_{k-1}$ denotes the usual Legendre polynomial, and $\Int(S_{Q})$ denotes the bounded component of $\H \setminus S_{Q}$. For $k = 1$ the function $\mathcal{P}_{1,A}(\tau)$ is defined analogously, but the first line has to be omitted. If $\tau \in \H$ does lie on one of the semi-circles $S_{Q}$ for $Q \in [A]$, we define the value of $\mathcal{P}_{k,A}$ at $\tau$ by the average value
\[
\mathcal{P}_{k,A}(\tau) := \lim_{\varepsilon \to 0}\frac{1}{2}\big(\mathcal{P}_{k,A}(\tau+i\varepsilon)+\mathcal{P}_{k,A}(\tau-i\varepsilon)\big).
\]
Note that the sum in the second line of \eqref{eq def local polynomial} is finite and $\mathcal{P}_{k,A}(\tau)$ has discontinuities along the semi-circles $S_{Q}$ for $Q \in [A]$. From the properties of $\zeta_{\Gamma,A}(s)$ given in Section~\ref{section zeta functions} it easily follows that the special values 
\[
|d|^{\frac{k-1}{2}}\mathcal{P}_{k,A}(\tau_{P})
\]
at CM points $\tau_{P} \in \H$ associated to positive definite forms $P \in \mathcal{Q}_{d}$ are rational numbers.

Our first rationality result concerns linear combinations of cycle integrals of a fixed $f_{k,P}$.

\begin{theorem}\label{theorem traces rationality}
	Let $\mathcal{Q}$ be a finite family of indefinite quadratic forms 
	%$A \in \mathcal{Q}_{D_{A}}$ 
	of non-square discriminants and $a_{A} \in \Z$ for $A \in \mathcal{Q}$ such that $\sum_{A \in \mathcal{Q}}a_{A}f_{k,A}= 0$ in $S_{2k}(\Gamma)$. Furthermore, let $P \in \mathcal{Q}_{d}$ be a positive definite quadratic form of discriminant $d$. Then we have the formula
\[
	\sum_{A \in \mathcal{Q}}a_{A}\mathcal{C}(f_{k,P},A)	= \frac{|d|^{\frac{k-1}{2}} }{|\overline{\Gamma}_{P}|}\sum_{A \in \mathcal{Q}} a_{A} \mathcal{P}_{k,A}(\tau_P),
\]
where $\overline{\Gamma}_{P}$ is the stabilizer of $P$ in $\Gamma/\{\pm 1\}$.
	  In particular, this linear combination of cycle integrals is a rational number whose denominator is bounded only in $k$ and $N$.
\end{theorem}

%\begin{remark}
%In the proof of Theorem 1.6 of \cite{ANBS} it is shown that $\sum_{D>0}a_{-D}f_{k,D}= 0$ (for $N = 1$) whenever $\sum_{D\gg-\infty}a_{D}q^D$ is the Fourier expansion of a weight $\frac32-k$ weakly holomorphic modular form for $\Gamma_{0}(4)$ in the Kohnen plus space. Hence Theorem \ref{theorem traces rationality} generalizes Theorem 1.6 of \cite{ANBS}.
%\end{remark}

We would like to emphasize that the formula on the right-hand side can be evaluated exactly, giving the precise rational value of the linear combination of cycle integrals on the left-hand side.

The proof of Theorem~\ref{theorem traces rationality} uses the fact that the cycle integral $\mathcal{C}(f_{k,P},A)$ equals the special value at the CM point $\tau_{P}$ of (the iterated derivative of) a so-called locally harmonic Maass form $\mathcal{F}_{1-k,A}(\tau)$, see Corollary~\ref{corollary main identity}. This function was introduced by Bringmann, Kane, and Kohnen in \cite{bringmannkanekohnen}. They showed that $\mathcal{F}_{1-k,A}$ can be decomposed into a sum of a certain local polynomial (whose iterated derivative is $\mathcal{P}_{k,A}$), and holomorphic and non-holomorphic Eichler integrals of the cusp form $f_{k,A}$. Taking suitable linear combinations as in the theorem, one can achieve that the Eichler integrals cancel out, which yields the formula in Theorem~\ref{theorem traces rationality}. We refer to Section~\ref{section proof theorem traces rationality} for the details of the proof.

\begin{example}
Let $N = 1$. Since there are no non-trivial cusp forms of weight less than $12$ or weight $14$ for $\Gamma(1)$, the functions $f_{k,A}$ for $k \leq 5$ and $k=7$ vanish identically for every indefinite quadratic form $A$. Thus it follows from Theorem~\ref{theorem traces rationality} that the cycle integrals $\mathcal{C}(f_{k,P},A)$ are rational for $k \leq 5$ and $k=7$ for every choice of $P$ and $A$. This explains the rationality of the cycle integrals of $f_{2,-3}$ and $f_{4,-3}$ that we observed in the introduction.
In contrast, we have seen in the introduction that the cycle integrals $\mathcal{C}(f_{6,-3}, A)$ do not seem to be rational. This corresponds to the fact that $f_{6,A}$ is a cusp form of weight $12$ which does usually not vanish identically. However, using results of \cite{zagqf}, one can prove the relations
\[
2 f_{6,[1,1,-1]} + f_{6,[1,0,-2]}= 11 f_{6,[1,1,-1]} + f_{6,[1,1,-3]} = f_{6,[1,0,-1]} -f_{6,[1,1,-4]} =0.
\]
Now Theorem \ref{theorem traces rationality} asserts that for any positive definite quadratic form $P$, the corresponding linear combinations
\begin{align*}
2\mathcal{C}(f_{6,P}, [1,1,-1]) + \mathcal{C}(f_{6,P}, [1,0,-2]),& \\
11\mathcal{C}(f_{6,P}, [1,1,-1]) + \mathcal{C}(f_{6,P}, [1,1,-3]),& \\
 \mathcal{C}(f_{6,P}, [1,0,-2]) - \mathcal{C}(f_{6,P}, [1,1,-4]),&
\end{align*}
of cycle integrals of $f_{6,P}$ are rational numbers. For example, for $f_{6,[1,1,1]} = f_{6,-3}$ we have
\begin{align*}
2\mathcal{C}(f_{6,-3}, [1,1,-1]) + \mathcal{C}(f_{6,-3}, [1,0,-2]) &=696,\\
11\mathcal{C}(f_{6,-3}, [1,1,-1]) + \mathcal{C}(f_{6,-3}, [1,1,-3]) &=2616,\\
\mathcal{C}(f_{6,-3}, [1,0,-2]) - \mathcal{C}(f_{6,-3}, [1,1,-4]) &=-11940.
\end{align*} 

\end{example}

Next, we consider cycle integrals of certain linear combinations of forms $f_{k,P}$ over a single geodesic. Following \cite{grosszagier}, we call a sequence $\underline{\lambda} = (\lambda_{m})_{m=1}^{\infty} \subset \Z$ of integers a \emph{relation} for $S_{2k}(\Gamma)$ if 
	\begin{enumerate}
		\item $\lambda_{m} = 0$ for almost all $m$,
		\item $\sum_{m=1}^{\infty}\lambda_{m}c_{f}(m) = 0$ for every cusp form $f(z) = \sum_{m=1}^{\infty}c_{f}(m)q^{m} \in S_{2k}(\Gamma)$, and
		\item $\lambda_{m} = 0$ whenever $(m,N) > 0$.
	\end{enumerate}
%	Equivalently, this means that there exists a weakly holomorphic modular form of weight $2-2k$ for $\Gamma_{0}(N)$ whose Fourier expansion at $\infty$ has the form $\sum_{n=1}^{\infty}\lambda_{n}q^{-n} + O(1)$ and which is holomorphic at the other cusps. 
	For a meromorphic function $f$ on $\H$ which transforms like a modular form of weight $2k$ for $\Gamma$ we define its \emph{Hecke translate} corresponding to a relation $\underline{\lambda}$ by
\[
f|T_{\underline{\lambda}} := \sum_{m=1}^{\infty}\lambda_{m}f|T_{m},
\]
where $T_{m}$ denotes the usual $m$-th Hecke operator of level $N$, see \eqref{eq definition Tn}. Bengoechea \cite{bengoecheapaper} showed that the Fourier coefficients of $f_{k,P}|T_{\underline{\lambda}}$ are algebraic multiples of $\pi^{k-1}$ for every positive definite quadratic form $P$ and every relation $\underline{\lambda}$ for $S_{2k}(\Gamma)$. We obtain a rationality result for the cycle integrals of $f_{k,P}|T_{\underline{\lambda}}$.

\begin{theorem}\label{theorem traces rationality 2}
	Let $A \in \mathcal{Q}_{D}$ be an indefinite quadratic form of non-square discriminant $D$. Furthermore, let $P \in \mathcal{Q}_{d}$ be a positive definite quadratic form of discriminant $d$ and let $\underline{\lambda} = (\lambda_{m})$ be a relation for $S_{2k}(\Gamma)$. Then we have the formula
	\begin{align*}
	\mathcal{C}\left(f_{k,P}|T_{\underline{\lambda}},A\right) &= \frac{|d|^{\frac{k-1}{2}} }{|\overline{\Gamma}_{P}|}\sum_{m \geq 1} \lambda_{m}m^{k-1}\sum_{\substack{\alpha \delta =m \\ \delta>0}}\sum_{\beta\!\!\!\!\!\pmod{\delta}} \mathcal{P}_{k,A}\left(\frac{\alpha \tau_P+\beta}{\delta}\right).
	  \end{align*}
	  In particular, the cycle integrals of $f_{k,P}|T_{\underline{\lambda}}$ are rational numbers whose denominators are bounded only in $k$ and $N$.
\end{theorem}

The idea of the proof is similar as for Theorem~\ref{theorem traces rationality}. See Section~\ref{section proof theorem traces rationality 2} for the details.

\begin{example}\label{Heckexp}
Let $N = 1$. For $k=6$, we have the relation $\underline{\lambda} = (24,1,0,0,\dots)$ for $S_{12}$. By Theorem~\ref{theorem traces rationality 2} applied to $f_{6,-3} = f_{6,[1,1,1]}$, the function
\[
f_{6,-3}|T_{\underline{\lambda}} =f_{6,-3}|T_2 + 24 f_{6,-3} =  f_{6,-12} -8 f_{6,-3}
\]
has rational cycle integrals. Here we used the action of $T_{p}$ on $f_{k,D}$ as stated in \cite{bengoecheapaper}. Indeed, we have  
\begin{align*}
\renewcommand{\arraystretch}{1.2}
\begin{array}{|c||c|c|c|c|c|c|}
\hline
A & [1,1,-1] & [1,0,-2] & [1,1,-4] & [1,0,-6] & [1,1,-7] & [1,1,-8] \\
\hline\hline
\mathcal{C}(f_{6,-3}|T_{\underline{\lambda}},A) & 5952 & 44112 & 1128096 & 1186056 & 2349504 & 4070304\\
\hline
\end{array}
\end{align*} 
Similarly, for $f_{6,-7} = f_{6,[1,1,2]}$, we have
\[
f_{6,-7}|T_{\underline{\lambda}} = f_{6,-7}|T_{2}+24f_{6,-7} = f_{6,-28} +56 f_{6,-7}
\]
and
\begin{align*}
\renewcommand{\arraystretch}{1.2}
\begin{array}{|c||c|c|c|c|c|c|}
\hline
A & [1,1,-1] & [1,0,-3] & [1,1,-3] & [1,1,-4]& [1,1,-5] & [1,0,-6] \\
\hline\hline
\mathcal{C}(f_{6,-7}|T_{\underline{\lambda}},A) & 228704 & 2728656 & 7282240 & 17047968 & 15937488 & 26668656 \\
\hline
\end{array}
\end{align*}
\end{example}

%Given that the combinations $f_{k,P,\underline{\lambda}}$ have (up to a power of $\pi$) rational Fourier coefficients and rational cycle integrals, numerical evidence {\bf (Include)}, and the fact that rational Fourier coefficients imply rational cycle integrals in weight $2$ {\bf (Reference)}, we are led to the following conjecture. 

Finally, we consider cycle integrals of linear combinations of the forms $f_{k,D}$ and their twisted analogs $f_{k,\Delta,\delta}$, which we define now. For simplicity, we now assume that $N$ is odd and square-free. Let $k \geq 1$, let $\Delta$ be a discriminant with $(-1)^{k}\Delta > 0$, and let $\delta$ be a fundamental discriminant with $(-1)^{k}\delta < 0$, such that $\delta$ is a square modulo $4N$. Let $\chi_{\delta}$ be the generalized genus character on $\mathcal{Q}_{\Delta \delta}$ as defined in \cite{grosskohnenzagier}. For $k \geq 1$ we define the twisted function
\[
f_{k,\Delta,\delta}(z) := \sum_{P \in \mathcal{Q}_{\Delta \delta}/\Gamma}\chi_{\delta}(P)f_{k,P}(z).
\]
Then $f_{k,\Delta,\delta}$ is a meromorphic modular form of weight $2k$ for $\Gamma$. Suppose that 
\[
F(\tau) = \sum_{m \gg -\infty}c_{F}(m)q^m
\] 
is a weakly holomorphic modular form of weight $\frac{3}{2}-k$ for $\Gamma_{0}(4N)$ satisfying the Kohnen plus space condition, such that the Fourier coefficients $c_{F}(m)$ are rational for all $m < 0$. We will show in Proposition~\ref{RatCoef} that the Fourier coefficients of the meromorphic modular form
\begin{align}\label{eq twisted sum}
\sum_{(-1)^{k}\Delta > 0}c_{F}(-|\Delta|)f_{k,\Delta,\delta}(z)
\end{align}
are algebraic multiples of $\pi^{k-1}$. Based on extensive numerical experiments, we arrived at the following conjecture.

\begin{conjecture}\label{RatCycint}
The function in \eqref{eq twisted sum} has rational cycle integrals if it has no poles on the cycle.
\end{conjecture}

It seems that the methods used to prove Theorem~\ref{theorem traces rationality} and Theorem~\ref{theorem traces rationality 2} are not suitable to prove the conjecture. In particular, numerical computations suggest that the cycle integrals of the function in \eqref{eq twisted sum} cannot be expressed in a simple way in terms of the functions $\mathcal{P}_{k,A}$. However, we are able to prove the conjecture in the case $k = 1$, using different methods.

\begin{theorem}\label{weight 2}
Conjecture~\ref{RatCycint} is true for $k = 1$.
\end{theorem}

The proof relies on the fact that the function $\pi i f_{1,\Delta,\delta}(z)dz$ is the canonical differential of the third kind for its residue divisor on the compactified modular curve $X_{0}(N)$. Together with a rationality criterion of Scholl \cite{scholl} for such differentials we obtain Theorem~\ref{weight 2}. We refer to Section~\ref{section proof weight 2} for the proof. We remark that, unfortunately, the proof does not yield finite rational formulas for the cycle integrals of the linear combination \eqref{eq twisted sum}.

\begin{example}
We give some numerical examples of Conjecture~\ref{RatCycint}. Let $N = 1$. If $k$ is odd, we can pick $\delta=1$, such that there is no twist. 
The first odd $k$ for which there are nontrivial weight $2k$ cusp forms is $k=9$ with $S_{18} = \C\Delta E_6$. The space $S_{18}$ is isomorphic to the Kohnen plus space of weight $9+\frac{1}{2}$ under the Shimura correspondence, and the latter space is spanned by the cusp form
\[
q^3 - 2q^4 - 16q^7 + 36q^8 + O(q^{11})
\]
This implies that there is a weakly holomorphic modular form with principal part $q^{-4} + 2q^{-3}+O(1)$ in the Kohnen plus space of weight $\frac32-9$. In this case, Conjecture~\ref{RatCycint} predicts that the linear combination
\[
g := f_{9,-4}+ 2f_{9,-3}
\]
has rational cycle integrals. One can easily see that, since $k$ is odd, we have $\mathcal{C}(g,A)=0$ whenever the form $A$ is $\Gamma(1)$-equivalent to $-A$. But for quadratic forms that are not equivalent to their negatives, we obtain numerically:
\begin{align*}
\renewcommand{\arraystretch}{1.2}
\begin{array}{|c||c|c|c|c|c|c|}
\hline A & [1,1,-5] & [1,0,-6] & [1,1,-8] & [1,0,-11] &[1,0,-14] & [1,1,-14] \\
\hline\hline
\mathcal{C}(g,A)  & 3343284 & 235476 & 4350060 &   116285048 & 255683332 & 254947680 \\
\hline
\end{array}
\end{align*}

If $k$ is even, we have to introduce a twist, since $\delta <0$.
Here we consider $k=6$ and $\delta =-3$.
The weight $6+\frac12$ cusp form corresponding to $\Delta \in S_{12}$ under the Shimura correspondence is given by
\[
q - 56q^4 +120q^5 -240q^8 +9q^9 + O(q^{12}),
\]
so for example the functions 
\[
g_{1}:=f_{6,4,-3} + 56 f_{6,1,-3} \qquad \text{and} \qquad  g_{2}:=f_{6,5,-3} -120 f_{6,1,-3}
\]
should have rational cycle integrals. Indeed, it is easy to check that the function $g_{1}$ coincides with $f_{6,-3}|T_{\underline{\lambda}}$ from Example \ref{Heckexp}, so it does have rational cycle integrals by Theorem~\ref{theorem traces rationality 2}. In constrast, $g_{2}$ cannot be obtained by acting with Hecke operators, since $-15$ is squarefree. 
However, for this function we obtain numerically:
\begin{align*}
	\renewcommand{\arraystretch}{1.5}
	\begin{array}{|c||c|c|c|c|c|c|c|}
	\hline
	A & [1,1,-1] & [1,0,-2] & [1,1,-3]& [1,1,-4] & [1,1,-5] & [1,1,-7] & [1,1,-8] \\
	\hline\hline
	\mathcal{C}(g_{2},A) & -51012 & -126816 & 57876 & -2108352 & 134946 & 3813312 & -7458750 \\
	\hline
	\end{array}
\end{align*}
\end{example}

The work is organized as follows. In Section \ref{section preliminaries} we introduce the necessary functions and notation. Then we relate the cycle integrals $\mathcal{C}(f_{k,P},A)$ to locally harmonic Maass forms in Section \ref{locally}. The proofs of Theorems~\ref{theorem traces rationality}, \ref{theorem traces rationality 2}, and \ref{weight 2} are given in the remaining sections.

\section*{Acknowledgments} We thank Jan Bruinier for insightful discussions on the proof of Theorem~\ref{weight 2}. Furthermore, we thank Kathrin Bringmann for helpful comments on an earlier draft of this paper.

\section{Preliminaries}\label{section preliminaries}

\subsection{Weight $2$ Eisenstein series}
%We define the harmonic weight $2$ Eisenstein series $E_{2,\Gamma}^*$ for $\Gamma=\Gamma_0(N)$ associated to the cusp $i\infty$ to be the analytic continuation to $s=0$ of 
%$$
%E_{2,\Gamma,s}(z):=\sum_{M\in\Gamma_\infty\backslash\Gamma}y^s|_2 M
%$$
For $z = x+iy \in \H$ we define the quasimodular weight $2$ Eisenstein series for $\Gamma=\Gamma_0(N)$ associated to the cusp $i\infty$ by the conditionally convergent series
\[
E_{2,\Gamma}(z) :=1 + \sum_{c\geq 1}\sum_{\substack{d\in\Z \\ \left(\begin{smallmatrix}*&*\\c&d\end{smallmatrix}\right)\in  \Gamma_\infty\setminus\Gamma }}1|_{2}M,
\]
where $\Gamma_{\infty} := \{\pm\left(\begin{smallmatrix}1 & n \\ 0 & 1 \end{smallmatrix} \right): n\in \Z\}$ and $\left(f|_{k}\left(\begin{smallmatrix}a & b \\ c & d  \end{smallmatrix} \right)\right)(z) := (cz +d)^{-k}f\big(\frac{az + b}{cz + d}\big)$ denotes the usual weight $k$ slash operator. The function $E_{2,\Gamma}$ has a non-holomorphic modular completion
\[
E_{2,\Gamma}^*(z) := -\frac{3}{\pi [\Gamma(1):\Gamma]y} + E_{2,\Gamma}(z)
\]
that has constant term $1$ at $i\infty$ and $0$ at all other cusps. The following lemma expresses $E_{2,\Gamma}^*$ in terms of the Eisenstein series for the full modular group and follows from eq.~(9) on p.~546 of 
\cite{grosskohnenzagier}.

\begin{lemma}\label{E2N}
We have 
\begin{align*}
E_{2,\Gamma}^*(z) &= \prod_{p|N }\left(1-p^{-2}\right)^{-1}\sum_{d|N}\frac{\mu(d)}{d^2}E_{2, \Gamma(1)}^*\left(\frac{N}{d}z\right) \\
&=-\frac{3}{\pi [\Gamma(1):\Gamma]y} + 1 - 24\prod_{p|N}\left(1-p^{-2}\right)^{-1}\sum_{d|N}\frac{\mu(d)}{d^2}\sum_{n\geq 1}\sigma\left(\frac{dn}{N}\right)q^n,
\end{align*}
where $\sigma$ denotes the divisor sum function and we set $\sigma(x):=0$ for $x\notin\Z$.  
\end{lemma}

\subsection{Petersson's Poincar\'e series}\label{section petersson poincare}

For $z = x+iy,\tau =u+iv \in \H$ and $k \in \Z$ with $k \geq 2$ we define Petersson's Poincar\'e series
\begin{align*}
H_{k}(z,\tau) &:=\sum_{M \in \Gamma}\left(\frac{(z-\tau)(z-\overline{\tau})}{v}\right)^{-k}\Biggl|_{2k, z} M =\sum_{M \in \Gamma}\left(\frac{(z-\tau)(z-\overline{\tau})}{v}\right)^{-k}\Biggl|_{0, \tau} M,
\end{align*}
which has weight $2k$ in $z$ for $\Gamma$ and weight $0$ in $\tau$ for $\Gamma$. Furthermore, it is meromorphic as a function of $z$, and an eigenfunction of the invariant Laplace operator $\Delta_{0}$ with eigenvalue $k(1-k)$ as a function of $\tau$ for $\tau$ not lying in the $\Gamma$-orbit of $z$. The series does not converge for $k=1$. However, we can apply Hecke's trick as in \cite{bringmannkaneweight0} and define for $\Re(s)>0$
\[
H_{1,s}(z,\tau) :=\sum_{M \in \Gamma}\left(\left(\frac{(z-\tau)(z-\overline{\tau})}{v}\right)^{-1}\left(\frac{|z-\tau||z-\overline{\tau}|}{vy}\right)^{-s}\right)\Biggl|_{0, \tau} M.
\]
One can show that $H_{1,s}(z,\tau)$ has an analytic continuation $H_{1}^*(z,\tau)$ to $s=0$. 
It is not meromorphic in $z$ anymore, but the function
\begin{equation}\label{H1E2}
H_{1}(z,\tau) := H_{1}^*(z,\tau)-2\pi E_{2,\Gamma}^*(z)
\end{equation} 
is a meromorphic modular form of weight $2$ in $z$ and a harmonic Maass form of weight $0$ in $\tau$ for $\Gamma$. The function $z\mapsto H_{1}(z,\tau)$ has a simple pole when $z$ is $\Gamma$-conjugate to $\tau$.

Similarly, we define for $k\geq 2$ and $\ell \in\Z$ the function
\begin{align*}
H_{k,\ell}(z,\tau) 
%&= \sum_{M \in \Gamma}j(M,z)^{-2k}\left(\frac{(Mz-\tau)(Mz-\overline{\tau})}{v}\right)^{-k-\ell}(Mz-\tau)^{2\ell} \\
&:= \sum_{M \in \Gamma}v^{k+\ell}\left((z-\tau)^{\ell-k}(z-\overline{\tau})^{-\ell-k}\right)\Bigl|_{2k, z} M \\
&= \sum_{M \in \Gamma}v^{k+\ell}\left((z-\tau)^{\ell-k}(z-\overline{\tau})^{-\ell-k}\right)\Bigl|_{-2\ell, \tau} M.
\end{align*}
It has weight $2k$ in $z$ and weight $-2\ell$ in $\tau$ for $\Gamma$, and it also behaves nicely under the raising and lowering operators 
\begin{align*}
R_{\kappa} := 2i\frac{\partial}{\partial \tau} + \kappa v^{-1}, \qquad L_{\kappa} := -2i v^{2}\frac{\partial}{\partial \overline{\tau}},
\end{align*}
which raise and lower the weight of an automorphic form of weight $\kappa$ by $2$, respectively. The following lemma can be checked by a direct computation.

\begin{lemma}\label{raislow}
	For $k \geq 2$ and $\ell \in \Z$ we have
	\begin{align*}
	R_{-2\ell,\tau}\left(H_{k,\ell}(z,\tau)\right) &= (k-\ell)H_{k,\ell-1}(z,\tau), \\
	L_{-2\ell,\tau}\left(H_{k,\ell}(z,\tau)\right) &= (k+\ell)H_{k,\ell+1}(z,\tau).
	\end{align*}
\end{lemma}

We are particularly interested in the function $H_{k,k-1}(z,\tau)$ (with $H_{1,0}(z,\tau):=H_{1}(z,\tau)$), which has weight $2k$ in $z$ and $2-2k$ in $\tau$. It is meromorphic in $z$ and harmonic in $\tau$ for $\tau$ not lying in the $\Gamma$-orbit of $z$, and as a function of $\tau$ it is bounded at the cusps (and vanishes at $i\infty$ if $k = 1$, compare Lemma 5.4 in \cite{bringmannkaneweight0}). Furthermore, by Lemma~\ref{raislow} it is related to $H_{k}(z,\tau)$ by
\begin{align}\label{eq Hk and Hkk-1}
R_{2-2k,\tau}^{k-1}\left(H_{k,k-1}(z,\tau)\right) = (k-1)!H_{k}(z,\tau),
\end{align}
where $R_{2-2k}^{k-1} := R_{-2} \circ \cdots \circ R_{2-2k}$ is an iterated version of the raising operator.

\subsection{Modular forms associated to quadratic forms}\label{section modular forms quadratic forms}
	
	In the introduction we defined the function $f_{k,Q_{0}}$ associated to an integral binary quadratic form $Q_{0}$ of discriminant $D \neq 0$ for $k\geq 2$. We briefly explain the definition for $k = 1$. For $s \in \C$ with $\Re(s) > 1$ we consider the series
	\begin{align*}
	f_{1,Q_{0},s}(z) := \frac{|D|^{\frac{s+1}{2}}}{2^s\pi}y^{s}\sum_{Q \in [Q_{0}]}Q(z,1)^{-1}|Q(z,1)|^{-s}.
	\end{align*}
	It converges absolutely and has a holomorphic continuation to $s = 0$. If $Q_{0} = A$ is indefinite and $D$ not a square, then 
	\[
	f_{1,A}(z) := f_{1,A,0}(z)
	\]
	is a cusp form of weight $2$ for $\Gamma$ (see \cite{grosskohnenzagier}, p. 517). 
	If $Q_{0} = P$ is positive definite, then it follows from the following lemma and \eqref{H1E2} that
	\begin{equation}\label{fPmero}
	f_{1,P}(z) := f_{1,P,0}(z) -\frac{2}{|\overline{\Gamma}_P|} E_{2, \Gamma}^{*}(z)  
	\end{equation}
	is a meromorphic modular form of weight $2$ for $\Gamma$.
	
%	We show that $f_{k,P}(z)$ can be written as a special value of the Petersson Poincar\'e series $H_{k}(z,\tau)$. 

\begin{lemma}\label{lemma fkP and Hk}
	For $k\geq 2$ and $z$ not lying in the $\Gamma$-orbit of the CM point $\tau_{P}$ we have
	\[
	f_{k,P}(z) = \frac{2^{k-1}|d|^{\frac{k-1}{2}}}{|\overline{\Gamma}_{P}|\pi}H_{k}(z,\tau_{P})
	\]
	and 
	$$
	f_{1,P,0}(z) = \frac{1}{|\overline{\Gamma}_{P}|\pi}H_{1}^*(z,\tau_{P}).
	$$
\end{lemma}

\begin{proof}
	We have the formula
	\[
	P(z,1) = \sqrt{|d|}\frac{(z-\tau_{P})(z-\overline{\tau}_{P})}{2\im(\tau_{P})}.
	\]
	Hence, for $k \geq 2$ we get
	\begin{align*}
	H_{k}(z,\tau_{P}) &= \sum_{M \in \Gamma}j(M,z)^{-2k}\left(\frac{(Mz-\tau_{P})(Mz-\overline{\tau}_{P})}{\im(\tau_{P})}\right)^{-k} \\
	&= 2^{-k}|d|^{\frac{k}{2}}\sum_{M \in \Gamma}j(M,z)^{-2k}P(Mz,1)^{-k} 
%	&= 2^{-k}|d|^{k/2}\sum_{M \in \Gamma}(P\circ M)(z,1)^{-k} \\
	= 2^{1-k}|d|^{\frac{1-k}{2}}|\overline{\Gamma}_{P}|\pi f_{k,P}(z).
	\end{align*}
	For $k=1$ we can show in the same way that $H_{1,s}(z,\tau_P)$ is a multiple of $f_{1,P,s}(z)$ and then use analytic continuation. 
\end{proof}

%In this section we give several different formulas for the Fourier expansion of $f_{k,P}$ for positive definite $P \in \mathcal{Q}_{d}$. We start with the Fourier expansion of $f_{k,P}$ as an infinite series involving Sali\'e sums and Bessel functions. The proof is similar to the proof of Proposition 2.2 from \cite{bengoechea}, so we omit it.

We will also need the Fourier expansion of $f_{k,P}$. The proof of the following formula is analogous to the proof of Proposition 2.2 from \cite{bengoecheapaper} (correcting a sign error), but additionally uses \eqref{fPmero} and Lemma \ref{E2N} in case that $k=1$.

\begin{proposition}\label{prop fkP Fourier expansion}
	For $k \geq 2$ and $z \in \H$ with $y > \sqrt{|d|}/2$ we have the Fourier expansion
	\[
	f_{k,P}(z) = \sum_{n \geq 1}c_{f_{k,P}}(n)e^{2\pi i n z},
	\]
	where
	\[
	c_{f_{k,P}}(n) = \frac{(-1)^{k}2^{k+\frac{1}{2}}\pi^{k}}{(k-1)!}|d|^{\frac{k}{2}-\frac{1}{4}}n^{k-\frac{1}{2}}\sum_{\substack{a \geq 1 \\ N \mid a}}a^{-\frac{1}{2}}S_{a,P}(n)I_{k-\frac{1}{2}}\left( \frac{\pi n\sqrt{|d|}}{a}\right),
	\]
	with the usual $I$-Bessel function and the exponential sum
	\[
	S_{a,P}(n) := \sum_{\substack{b \!\!\!\!\pmod{2a} \\ b^{2} \equiv d \!\!\!\!\pmod{4a} \\ \left[a,b,\frac{b^{2}-d}{4a}\right] \in [P]}} e\left(\frac{n b}{2a}\right).
	\]
	For $k = 1$ the formula is analogous, but we have to add 
	\[
	\frac{12}{|\overline{\Gamma}_P|}\prod_{p|N}\left(1-p^{-2}\right)^{-1}\sum_{d|N}\frac{\mu(d)}{d^2}\sigma\left(\frac{dn}{N}\right)
	\]	
	to $c_{f_{1,P}}(n)$, and we get a constant term $c_{f_{1,P}}(0) = -\frac{2}{|\overline{\Gamma}_P|}$.
\end{proposition}

\subsection{Zeta functions associated to indefinite quadratic forms}
\label{section zeta functions}

Let $A \in \mathcal{Q}_{D}$ be an indefinite quadratic form of non-square discriminant $D > 0$. We define the associated zeta function
\begin{equation}\label{zetadef}
\zeta_{\Gamma,A}(s) := \sum_{\substack{\left(\begin{smallmatrix}m & m_{0} \\ n & n_{0} \end{smallmatrix}\right) \in \Gamma_{A}\backslash \Gamma/\Gamma_{\infty} \\ A(m,n) > 0}}A(m,n)^{-s} =\sum_{\substack{(m,n) \in \Z^{2}/\Gamma_{A}^{t} \\ N| n, (m,n) = 1\\ A(m,n) > 0}}A(m,n)^{-s}.
\end{equation}
The series converges absolutely for $\Re(s) > 1$ and it only depends on the $\Gamma$-equivalence class of $A$. For $N = 1$ and a fundamental discriminant $D > 0$ the function $\zeta(2s)\zeta_{\Gamma(1),A}(s)$ is the usual zeta function of the ideal class in $\Q(\sqrt{D})$ associated to $A$. The above zeta function can also be written as a Dirichlet series
	\[
	\zeta_{\Gamma,A}(s) = \sum_{\substack{a > 0 \\ N \mid a}}\frac{n_{A}(a)}{a^{s}}, \qquad n_{A}(a) := \#\left\{b \!\!\!\! \pmod{2a}, b^{2}\equiv D \!\!\!\! \pmod{4a}, \left[a,b,\frac{b^{2}-D}{4a}\right] \in [A]\right\},
	\]
	compare \cite{zagierzetafunctions}, Proposition 3 (i). 
%	From this it is easy to see that
%\begin{align*}
%\zeta_{N,\overline{A}}(s) = \zeta_{N,A}(s) \quad \text{and} \quad \zeta_{N,-A}(s) = \zeta_{N,A^{*}}(s),
%\end{align*}
%where $\overline{A} = [a,-b,c]$ and $A^{*} = [-a,b,-c] = -\overline{A}$.

We first express $\zeta_{\Gamma,A}(s)$ in terms of zeta functions $\zeta_{\Gamma(1),A_{d}}(s)$ for level $1$ and suitable quadratic forms $A_{d}$.

\begin{lemma}\label{lemma zeta level lowering}
	For $\Re(s) > 1$ we have 
	\[
	\zeta_{\Gamma,A}(s) = N^{-s}\prod_{p \mid N}(1-p^{-2s})^{-1}\sum_{d \mid N}\frac{\mu(d)}{d^{s}}\left[\Gamma(1)_{A_{N/d}}:\Gamma_{0}(d)_{A_{N/d}}\right]\zeta_{\Gamma(1),A_{N/d}}(s),
	\]
	where $A_{d} := [a/d,b,dc]$ for $A = [a,b,c]$.
\end{lemma}

\begin{proof}
	Let $\zeta_{N}(s) = \sum_{\substack{\alpha = 1, (\alpha,N) = 1}}^{\infty}\alpha^{-s}$. We have
	\begin{align*}
	\zeta_{N}(2s)\zeta_{\Gamma,A}(s) &= \sum_{\substack{\alpha = 1 \\ (\alpha,N) = 1}}^{\infty}\sum_{\substack{(m,n) \in \Z^{2}/\Gamma_{0}(N)_{A}^{t} \\ N\mid n, (m,n) = 1  \\ A(m,n) > 0}}A(\alpha m, \alpha n)^{-s} = \sum_{\substack{(m,n) \in \Z^{2}/\Gamma_{0}(N)_{A}^{t} \\ N\mid n, (m,N) =1  \\ A(m,n) > 0}}A(m, n)^{-s} \\
	&= \sum_{d \mid N}\mu(d)\sum_{\substack{(m,n) \in \Z^{2}/\Gamma_{0}(N)_{A}^{t} \\ N\mid n, d \mid m  \\ A(m,n) > 0}}A(m, n)^{-s}= \sum_{d \mid N}\frac{\mu(d)}{d^{2s}}\sum_{\substack{(m,n) \in \Z^{2}/\Gamma_{0}(N)_{A}^{t} \\ \frac{N}{d}\mid n  \\ A(m,n) > 0}}A(m, n)^{-s}.
	\end{align*}
	If $(a,b)$ runs through $\Z^{2}/\Gamma_{0}(d)_{A_{N/d}}^{t}$ then $(m,n) = (a,\frac{N}{d}b)$ runs through $\Z^{2}/\Gamma_{0}(N)_{A}^{t}$ with $\frac{N}{d}\mid n$. Hence we obtain
	\begin{align*}
	\sum_{\substack{(m,n) \in \Z^{2}/\Gamma_{0}(N)_{A}^{t} \\ \frac{N}{d}\mid n, \, A(m,n) > 0}}A(m, n)^{-s} 
	&= \sum_{\substack{(m,n) \in \Z^{2}/\Gamma_{0}(d)_{A_{N/d}}^{t} \\ A(m,\frac{N}{d}n) > 0}}A(m, \tfrac{N}{d}n)^{-s} \\
	&= \left(\frac{N}{d}\right)^{-s}\sum_{\substack{(m,n) \in \Z^{2}/\Gamma_{0}(d)_{A_{N/d}}^{t} \\ A_{N/d}(m,n) > 0}}A_{N/d}(m,n)^{-s} \\
	&= \left[\Gamma(1)_{A_{N/d}}:\Gamma_{0}(d)_{A_{N/d}}\right]N^{-s}d^{s}\zeta(2s)\zeta_{\Gamma(1),A_{N/d}}(s).
	\end{align*}
	Using $\frac{\zeta(2s)}{\zeta_{N}(2s)} =\prod_{p \mid N}(1-p^{-2s})^{-1}$ we obtain the stated formula.
\end{proof}

The following result concerns the rationality of the special values of $\zeta_{\Gamma,A}(s)$ at positive integers.

\begin{proposition}\label{proposition zeta level lowering}
	The expression 
	\[
	D^{k-\frac{1}{2}}\left(\zeta_{\Gamma,A}(k)+(-1)^{k}\zeta_{\Gamma,-A}(k) \right)
	\]
	is rational for any $k \geq 2$, any non-square discriminant $D > 0$, and $N \in \N$.
\end{proposition}

\begin{proof}
	We set $\widehat{\zeta}_{A}(s) := \zeta(2s)\zeta_{\Gamma(1),A}(s)$. It is well known that for $k \geq 2$ and any non-square discriminant $D > 0$ we have the functional equation
\[
D^{k-\frac{1}{2}}\left(\widehat{\zeta}_{A}(k)+(-1)^{k}\widehat{\zeta}_{-A}(k)\right) = \frac{2^{2k-1}\pi^{2k}}{(k-1)!^{2}}\widehat{\zeta}_{A}(1-k),
\]
compare \cite{kohnenzagierrationalperiods}, p. 230. Furthermore, $\widehat{\zeta}_{A}(1-k)$ is rational by Theorem 8 in \cite{kohnenzagierrationalperiods}. Dividing by $\zeta(2k) = (-1)^{k+1}\frac{B_{2k}(2\pi)^{2k}}{2(2k)!}$ on both sides, we see that
\[
D^{k-\frac{1}{2}}\left(\zeta_{\Gamma(1),A}(k)+(-1)^{k}\zeta_{\Gamma(1),-A}(k)\right)
\]
is rational, too. Using Lemma~\ref{lemma zeta level lowering} we obtain the result for all $N \in \N$.
\end{proof}

\begin{remark}
%The denominator of the Bernoulli number $B_{2n}$ equals $\prod_{p-1|2n\atop \text{$p$ prime}}p$. 
It follows from the explicit formula in Theorem 8 of \cite{kohnenzagierrationalperiods} that the denominator of $\widehat{\zeta}_{A}(1-k)$ is bounded by a constant only depending on $k$, but not on $A$. This formula can also be used to evaluate $D^{k-\frac{1}{2}}\left(\zeta_{\Gamma,A}(k)+(-1)^{k}\zeta_{\Gamma,-A}(k) \right)$ explicitly as a rational number.
%$12^{k} k! \prod_{n=1}^k \prod_{p \geq 5 \text{ prime} \atop p-1|n }p$. 
%Hence the denominator of $D^{k-\frac{1}{2}}\left(\zeta_{N,A}(k)+(-1)^{k}\zeta_{N,-A}(k) \right)$ is bounded by
%$$
%\frac{N_{2k} 2^{k/2} (k-1)!^2 k!}{(2k)!}\prod_{n=1}^k \lcm\left(\prod_{p-1|2n\atop \text{$p>k-1$ prime}}p^2\right)
%$$ 
%where $N_{2k}$ is the numerator of $B_{2k}$. 
\end{remark}

Finally, we relate the expression from Proposition~\ref{proposition zeta level lowering} to the cycle integrals of the Eisenstein series $E_{2k,\Gamma}$ of weight $2k$ for the cusp $i\infty$ of $\Gamma$, which is normalized such that its Fourier expansion at $i\infty$ has constant term $1$. The following result can be proven by a similar computation as on pp. 240--241 of \cite{kohnenzagierrationalperiods}.

\begin{proposition}
	Let $k\geq 2$ and let $A \in \mathcal{Q}_{D}$ be an indefinite quadratic form of non-square discriminant $D > 0$. Then
	\[
	\mathcal{C}(E_{2k,\Gamma},A) = (-1)^{k}\frac{(k-1)!^{2}}{(2k-1)!}D^{k-\frac{1}{2}}\left(\zeta_{\Gamma,A}(k)+(-1)^{k}\zeta_{\Gamma,-A}(k) \right).
	\]
\end{proposition}

Although we will not use this formula in the proofs of our main results, we decided to include it since it gives an interesting interpretation of the expression from Proposition~\ref{proposition zeta level lowering}.

\subsection{Cycle integrals of meromorphic modular forms}\label{section cycle integrals}

Let $A = [a,b,c] \in \mathcal{Q}_{D}$ be an indefinite quadratic form of non-square discriminant $D > 0$. Then the set
\[
S_{A} :=\{z  \in \H: a|z|^{2}+b\Re(z)+c = 0\}
\] 
is a semi-circle centered at the real line. Let $f: \H \to \C$ be a meromorphic function which transforms like a modular form of weight $2k$ for $\Gamma$. If the poles of $f$ do not meet the semi-circle $S_{A}$, then we define the cycle integral of $f$ along the closed geodesic $c_{A} = \Gamma_{A}\backslash S_{A}$ by
\[
\mathcal{C}(f,A) := \int_{c_{A}}f(z)A(z,1)^{k-1}dz, 
\]
where $\Gamma_{A}$ denotes the stabilizer of $A$ in $\Gamma$. It only depends on the $\Gamma$-equivalence class of $A$. If some poles of $f$ do lie on $S_{A}$, we modify $S_{A}$ by circumventing these poles and all of their $\Gamma$-translates on small arcs of radius $\varepsilon > 0$ above and below the poles. Thereby we obtain two paths $S_{A,\varepsilon}^{+}$ and $S_{A,\varepsilon}^{-}$ and corresponding geodesics $c_{A,\varepsilon}^{+}$ and $c_{A,\varepsilon}^{-}$ which avoid the poles of $f$. We define the regularized cycle integral of $f$ along $c_{A}$ by the Cauchy principal value
\begin{align}\label{eq Cauchy principal value}
\mathcal{C}(f,A) := \lim_{\varepsilon \to 0}\frac{1}{2}\left(\int_{c_{A,\varepsilon}^{+}}f(z)A(z,1)^{k-1}dz+\int_{c_{A,\varepsilon}^{-}}f(z)A(z,1)^{k-1}dz\right).
\end{align}
Note that, since $f$ is meromorphic, the integrals on the right-hand side are actually independent of $\varepsilon$ for $\varepsilon > 0$ small enough, so the limit exists. 

\subsection{Maass Poincar\'e series}

Throughout this section we let $N$ be odd and square-free.
%\footnote{We make this assumption since we work with scalar valued modular forms of half-integral weight for $\Gamma_{0}(4N)$ for simplicity. One could achieve greater generality by working with vector valued modular forms for the Weil representation.} 
One can construct harmonic Maass form of half-integral weight as special values of Maass Poincar\'e series, see \cite{millerpixton}, for example. In this way, one obtains for every integer $k \geq 1$ and $n < 0$ with $(-1)^{k}n \equiv 0,3 \pmod 4$ a harmonic Maass form $\mathcal{P}_{\frac{3}{2}-k,n}(\tau)$ of weight $\frac{3}{2}-k$ for $\Gamma_{0}(4N)$ which satisfies the Kohnen plus space condition, whose Fourier expansion at $i\infty$ starts with $q^{-|n|}+O(1)$, and which is bounded at the other cusps.

The holomorphic part of $\mathcal{P}_{\frac{3}{2}-k,n}$ has a Fourier expansion of the shape
\[
\mathcal{P}_{\frac{3}{2}-k,n}^+ (\tau) = q^{-|n|} + \sum_{\substack{m \geq 0 \\ (-1)^{k}m \equiv 0,3 \!\!\!\!\pmod{4}}} 
c^+ _{\mathcal{P}_{\frac{3}{2}-k,n}}(m)q^m,
\]
whose coefficients of positive index are given as follows.

\begin{theorem}[Theorem 2.1 in \cite{millerpixton}]\label{theorem Poincare Fourier expansion}
\item Let $n<0$ and $m>0$ with $(-1)^{k}n,(-1)^{k}m\equiv 0,3\pmod{4}$. Then 
\begin{align*}
c_{\mathcal{P}_{\frac{3}{2}-k,n}}^{+}(m) &= 
-(-1)^{\left\lfloor\frac{k}{2}\right\rfloor}\pi\sqrt{2} \left(\frac{m}{|n|}\right)^{\frac14 -\frac{k}{2}}\sum_{\substack{a >0 \\ N \mid a}}\frac{K^{+}((-1)^{k+1}n, (-1)^{k+1}m,a)}{a}I_{k-\frac12}\left(\frac{\pi\sqrt{m|n|}}{a}\right),
\end{align*}
with the half-integral weight Kloosterman sum
\[
K^{+}(m,n,a) := \frac{1-i}{4}\left(1+\left(\frac{4}{a}\right)\right)\sum_{\nu\!\!\!\!\pmod{4a}^*}\left(\frac{4a}{\nu}\right)\left(\frac{-4}{\nu}\right)^{\frac12}e\left(\frac{m
\nu+n\overline{\nu}}{4a}\right).
\]
\end{theorem}

Let $\Delta,\delta \in \Z$ be discriminants and assume that $\delta$ is fundamental. For $a,n \in \Z$ we consider the \emph{Sali\'e sum} 
\[
S_{a,\Delta, \delta}(n):=\sum_{\substack{b \!\!\!\!\pmod{2a} \\ b^2 \equiv \delta\Delta \!\!\!\!\pmod{4a}}}\chi_\delta\left(\left[a,b,\frac{b^2-\delta\Delta}{4a}\right]\right)e\left(\frac{nb}{2a}\right),
\]
where $\chi_{\delta}$ is the generalized genus character of $\mathcal{Q}_{\Delta\delta}$ as defined in \cite{grosskohnenzagier}.
It is related to the half-integral weight Kloosterman sum by the following formula.
\begin{proposition}[Proposition 3 in \cite{dit}]\label{prop salie sum}
Let $\Delta,\delta \in \Z$ be discriminants and assume that $\delta$ is fundamental. Then for $a,n \in \Z$ we have the identity
\begin{align*}
S_{a,\Delta, \delta}(n)&= \sum_{m|(n,a)}\left(\frac{\delta}{m}\right)\sqrt{\frac{m}{a}}
 K^{+}\left(\Delta,\frac{n^2}{m^2}\delta,\frac{a}{m}\right).
\end{align*}
\end{proposition}

\section{Locally harmonic Maass forms}\label{locally}

The key to proving Theorems~\ref{theorem traces rationality} and \ref{theorem traces rationality 2} is relating the cycle integrals of $f_{k,P}$ to certain {\it locally harmonic Maass forms} introduced by Bringman, Kane, and Kohnen in \cite{bringmannkanekohnen} and H\"ovel in \cite{hoevel}. Namely, for $k\geq 2$, $\tau = u+iv \in \H$, and an indefinite quadratic form $A \in \mathcal{Q}_{D}$ of non-square discriminant $D > 0$, these are defined by the series
\begin{equation}\label{curlyfexp}
\mathcal{F}_{1-k, A}(\tau) : = \frac{(-1)^{k}D^{\frac{1}{2}-k}}{\binom{2k-2}{k-1}\pi}\sum_{Q \in [A]}\sgn(Q_{\tau})Q(\tau,1)^{k-1}\psi\left( \frac{Dv^{2}}{|Q(\tau,1)|^{2}}\right),
\end{equation}
where $Q_{\tau} := \frac{1}{v}(a|\tau|^{2}+bu + c)$ and
\[
\psi(v) :=\frac{1}{2}\beta\left(v;k-\frac{1}{2},\frac{1}{2}\right) = \frac{1}{2}\int_{0}^{v}t^{k-\frac{3}{2}}(1-t)^{-\frac{1}{2}}dt
\]
is a special value of the incomplete $\beta$-function. For $k=1$ one can define a weight $0$ analogue $\mathcal{F}_{0,A}$ of \eqref{curlyfexp} using the Hecke trick as in \cite{brikavia, ehlenguerzhoykanerolen}. By adding a suitable constant we can normalize $\mathcal{F}_{0,A}$ such that it vanishes at $i\infty$.

The function $\mathcal{F}_{1-k,A}$ transforms like a modular form of weight $2-2k$ for $\Gamma$, is harmonic on $\H \setminus \bigcup_{Q\in[A]}S_{Q}$ and bounded at the cusps, and has discontinuities along the semi-circles $S_{Q}$ for $Q \in [A]$. Its value at a point $\tau$ lying on $S_{Q}$ for $Q \in [A]$ is given by the average value
\begin{align}\label{eq average value}
\mathcal{F}_{1-k,A}(\tau) = \lim_{\varepsilon \to 0}\frac{1}{2}\big(\mathcal{F}_{1-k,A}(\tau+i\varepsilon)+\mathcal{F}_{1-k,A}(\tau-i\varepsilon)\big).
\end{align}
 Furthermore, outside the singularities $\mathcal{F}_{1-k,A}$ is related to the cusp form $f_{k,A} \in S_{2k}(\Gamma)$ by the differential equations
\begin{align*}
\xi_{2-2k}(\mathcal{F}_{1-k,A}) &= (-1)^{k}\frac{D^{\frac{1}{2}-k}}{\binom{2k-2}{k-1}}f_{k,A}, \\
\mathcal{D}^{2k-1}(\mathcal{F}_{1-k,A}) &= (-1)^{k+1}D^{\frac{1}{2}-k}\frac{(k-1)!^{2}}{(4\pi)^{2k-1}}f_{k,A},
\end{align*}
where $\xi_{\kappa} := 2iv^{\kappa}\overline{\frac{\partial}{\partial \overline{\tau}}}$ and $\mathcal{D} := \frac{1}{2\pi i}\frac{\partial}{\partial \tau}$. Note that our normalization of $f_{k,A}$ differs from the one used in \cite{bringmannkanekohnen}, which explains the different constants in the above differential equations.

Recall that the non-holomorphic and holomorphic \emph{Eichler integrals} of a cusp form $f = \sum_{n\geq 1}c_{f}(n)q^{n} \in S_{2k}(\Gamma)$ are defined by
\begin{align*}
f^{*}(\tau) := (-2i)^{1-2k}\int_{-\overline{\tau}}^{i\infty}\overline{f(-\overline{z})}(z+\tau)^{2k-2}dz,\qquad\mathcal{E}_{f}(\tau) := \sum_{n\geq 1}\frac{c_{f}(n)}{n^{2k-1}}q^{n}.
\end{align*}
They satisfy
\[
\xi_{2-2k}(f^{*}) = f, \qquad \mathcal{D}^{2k-1}(f^{*}) = 0, \qquad \xi_{2-2k}(\mathcal{E}_{f}) = 0, \qquad \mathcal{D}^{2k-1}(\mathcal{E}_{f}) = f.
\]
The following decomposition of $\mathcal{F}_{1-k,A}$ was derived by Bringmann, Kane and Kohnen for $N = 1$ and $k\geq 2$ in \cite{bringmannkanekohnen}, but the same methods work for all $N \geq 1$ and $k = 1$ (see also \cite{ehlenguerzhoykanerolen} for $k = 1$).

\begin{theorem}[Theorem 7.1 of \cite{bringmannkanekohnen}]\label{local}
For $\tau$ not lying on any of the semi-circles $S_{Q}$ for $Q \in [A]$ we have
\[
\mathcal{F}_{1-k,A}(\tau) = P_{1-k,A}(\tau) + (-1)^{k}\frac{D^{\frac{1}{2}-k}}{\binom{2k-2}{k-1}}f^{*}_{k,A}(\tau)+(-1)^{k+1}D^{\frac{1}{2}-k}\frac{(k-1)!^{2}}{(4\pi)^{2k-1}}\mathcal{E}_{f_{k,A}}(\tau),
\]
where $P_{1-k,A}(\tau)$ is locally a polynomial of degree at most $2k-2$. More precisely, it is a polynomial on each connected component of $\H \setminus \bigcup_{Q \in [A]}S_{Q}$, which is given by
\[
P_{1-k,A}(\tau) := c_{k}(A) + (-1)^{k-1}2^{2-2k}D^{\frac{1}{2}-k}\sum_{\substack{Q = [a,b,c] \in [A] \\\tau \in \Int(S_{Q})}}\sgn(a)Q(\tau,1)^{k-1},
\]
where $c_{1}(A) :=0 $ and 
\[
c_{k}(A):= -\frac{\zeta_{\Gamma,A}(k)+(-1)^{k}\zeta_{\Gamma,-A}(k)}{2^{2k-2}(2k-1)\binom{2k-2}{k-1}}
\]
for $k \geq 2$, and $\Int(S_{Q})$ denotes the bounded component of $\H\setminus S_{Q}$.
\end{theorem}

The main goal of this section is to show that $\mathcal{F}_{1-k,A}(\tau)$ can be written as a cycle integral of Petersson's Poincar\'e series $H_{k,k-1}(z,\tau)$.

\begin{theorem}\label{theorem locally harmonic maass form identity}
%For $\tau$ not lying on any of the semi-circles $S_{Q}$ for $Q \in [A]$ we have
We have
\begin{align*}
\mathcal{F}_{1-k,A}(\tau) = \frac{D^{\frac12-k}}{2\pi}\mathcal{C}\left( H_{k,k-1}(\cdot,\tau), A\right).
\end{align*}
If $\tau$ lies on a semi-circle $S_{Q}$ for $Q \in [A]$ the left-hand side has to interpreted as the average value \eqref{eq average value}, and the cycle integral on the right-hand side is defined as the Cauchy principal value \eqref{eq Cauchy principal value}.
%where $\overline{[a,b,c]} = [a,-b,c]$.
\end{theorem}

Before we come to the proof of the theorem we state an important corollary, which immediately follows from Theorem~\ref{theorem locally harmonic maass form identity} together with Lemma~\ref{lemma fkP and Hk} and the identity \eqref{eq Hk and Hkk-1}.

\begin{corollary}\label{corollary main identity} We have
\[
\mathcal{C}\left(f_{k,P},A\right) = \frac{2^{k}|d|^{\frac{k-1}{2}}D^{k-\frac12}}{(k-1)! \,|\overline{\Gamma}_{P}|}R_{2-2k}^{k-1}\left(\mathcal{F}_{1-k,A}\right)(\tau_{P}),
\] 
where $R_{2-2k}^{k-1}$ denotes the iterated raising operator defined in Section~\ref{section petersson poincare}.
\end{corollary}

Note that a harmonic function on $\H$ which transforms like a modular form of weight $2-2k$ and is bounded at the cusps has to be a constant (and therefore vanishes if $k > 1$). Hence, in order to prove Theorem~\ref{theorem locally harmonic maass form identity} in the case that $\tau$ does not lie on $S_{Q}$ for $Q \in [A]$ it suffices to show that both sides in the theorem have the same singularities on $\H$ and are bounded at the cusps (and vanish at $i\infty$ if $k = 1$).

 We say that a function $f$ has a \emph{singularity of type $g$} at a point $\tau_{0}$ if there exists a neighbourhood $U$ of $\tau_{0}$ such that $f$ and $g$ are defined on a dense subset of $U$ and $f-g$ can be extended to a harmonic function on $U$. For example, Theorem~\ref{local} shows that the function $\mathcal{F}_{1-k,A}(\tau)$ has a singularity of type
 \begin{align*}
 (-1)^{k}2^{1-2k}D^{\frac{1}{2}-k} \sum_{\substack{Q = [a,b,c]\in [A]\\ \tau_0 \in S_Q}}\sgn(Q_{\tau})Q(\tau,1)^{k-1}
 \end{align*}
 at each point $\tau_{0} \in \H$, which easily follows from the fact that $\tau \in \Int(S_{Q})$ is equivalent to $\sgn(a)\sgn(Q_{\tau}) < 0$. 
  
\begin{lemma}\label{jumps}
The function $\mathcal{C}( H_{k,k-1}(\cdot,\tau), A)$ is harmonic on $\H\setminus \bigcup_{Q \in [A]}S_{Q}$ and bounded at the cusps. For $k = 1$ it vanishes at $i\infty$. At a point $\tau_{0} \in \H$ it has a singularity of type
	\[
(-1)^{k}2^{2-2k}\pi \sum_{\substack{Q = [a,b,c] \in [A] \\ \tau_0 \in S_Q}}\sgn(Q_{\tau})Q(\tau,1)^{k-1}.
	\]
\end{lemma}

\begin{proof}
Since the function $\tau \mapsto H_{k,k-1}(z,\tau)$ is harmonic on $\H \setminus \Gamma z$, the function $\mathcal{C}(H_{k,k-1}(\cdot,\tau), A)$ is harmonic on $\H \setminus \bigcup_{Q \in [A]}S_{Q}$. Moreover, $\mathcal{C}( H_{k,k-1}(\cdot,\tau), A)$ is bounded at the cusps (and vanishes at $i\infty$ if $k = 1$) because the same is true for $\tau \mapsto H_{k,k-1}(z,\tau)$.

To determine the singularities, we keep $\tau_{0}\in\H$ fixed and consider the function
%$$
%\lim_{\varepsilon\rightarrow 0}\left(\mathcal{C}(H_{k,k-1}(\cdot,\tau_0+i\varepsilon), A)-\mathcal{C}(H_{k,k-1}(\cdot,\tau_0-i\varepsilon), A)\right).
%$$
%For this we write 
%$$
\begin{align*}
G_{\tau_0}(z, \tau) &:= \sum_{\substack{Q\in [A] \\ \tau_0 \in S_Q}}\sum_{M\in \Gamma_{Q}} \left(\frac{v^{2k-1}}{(z-\tau)(z-\overline{\tau})^{2k-1}}\right)\Big|_{2k, z} M.  
%&= \sum_{\substack{Q\in [A] \\ \tau_0 \in S_Q}}\sum_{M\in \Gamma_{Q}}  \left(\frac{v^{2k-1}}{(z-\tau)(z-\overline{\tau})^{2k-1}}\right)\Big|_{2-2k, \tau} M.
\end{align*}
Note that the sum over $Q \in [A]$ with $\tau_{0} \in S_{Q}$ is finite, and the group $\overline{\Gamma}_{Q}$ is infinite cyclic. It is not hard to show that the series converges absolutely and locally uniformly for all $k \geq 1$, and is meromorphic in $z$ and harmonic in $\tau$ for $\tau$ not lying in the $\Gamma$-orbit of $z$. We split the cycle integral into 
\[
\mathcal{C}(H_{k,k-1}(\cdot,\tau), A)  =\mathcal{C}(H_{k,k-1}(\cdot,\tau) - G_{\tau_0}(\cdot,\tau), A) + \mathcal{C}( G_{\tau_0}(\cdot, \tau), A).
\] 
The function 
%$$
%z\mapsto H_{k,k-1}(z,\tau) - G_{\tau_0}(z,\tau)
%$$
%is meromorphic and has no singularities on the half-circle $S_A$ and the function 
\[
\tau \mapsto \mathcal{C}(H_{k,k-1}(\cdot,\tau) - G_{\tau_{0}}(\cdot,\tau), A)
\]
is harmonic in a neighborhood of $\tau_0$. For the second summand we compute for any $\tau\notin \Gamma S_A$
\begin{align*}
\mathcal{C}( G_{\tau_0}(\cdot, \tau), A) &= \int_{c_A}\sum_{\substack{Q\in [A]\\ \tau_0 \in S_Q}} \sum_{M\in \Gamma_Q} \left(\left(\frac{v^{2k-1}}{(z-\tau)(z-\overline{\tau})^{2k-1}}\right)\Big|_{2k, z} M \right) A(z,1)^{k-1}dz \\
&=2\sum_{\substack{Q\in [A]\\ \tau_0 \in S_Q}}\int_{S_Q}\frac{v^{2k-1}}{(z-\tau)(z-\overline{\tau})^{2k-1}} Q(z,1)^{k-1}dz.
\end{align*} 
Note that the integrand is meromorphic in $z$. The integral is oriented counterclockwise if $a>0$ and clockwise if $a<0$. We complete $S_Q$ to a closed path by adding the horizontal line connecting the two real endpoints $w < w'$ of $S_Q$. 
The function
\[
\tau\mapsto \int_{w}^{w'}\frac{v^{2k-1}}{(x-\tau)(x-\overline{\tau})^{2k-1}} Q(x,1)^{k-1}dx
\]
is harmonic on $\H$, so it does not contribute to the singularity. From the residue theorem we obtain that the integral over the closed path equals $0$ if $\tau\notin \text{Int}(S_Q)$ and 
\begin{align*}
2\pi i \sgn(a) \operatorname{Res}_{z=\tau} \left(\frac{v^{2k-1}}{(z-\tau)(z-\overline{\tau})^{2k-1}} Q(z,1)^{k-1}\right) = (-1)^{k-1}2^{2-2k}\pi \sgn(a)Q(\tau,1)^{k-1}
\end{align*}
if $\tau\in \text{Int}(S_Q)$. This yields the claimed singularity.
%
%If we choose $\varepsilon$ small enough, we have $\tau_0+i\varepsilon\notin\text{Int}(S_Q)$ and $\tau_0-i\varepsilon\in\text{Int}(S_Q)$ for every $Q\in [A]$ with $\tau_0 \in S_Q$. This implies 
%\begin{align*}
%\lim_{\varepsilon\rightarrow 0}&\left(\mathcal{C}(H_{k,k-1}(\cdot,\tau_0+i\varepsilon), A)-\mathcal{C}(H_{k,k-1}(\cdot,\tau_0-i\varepsilon), A)\right) \\
%&= \lim_{\varepsilon\rightarrow 0}\left(\mathcal{C}(G_{\tau_0}(\cdot,\tau_0+i\varepsilon), A)-\mathcal{C}(G_{\tau_0}(\cdot,\tau_0-i\varepsilon), A)\right)\\
%&=(-1)^{k-1}2^{2-2k}\pi \sum_{Q\in [A]\atop \tau_0 \in S_Q}\sgn(a_Q)Q(\tau_0,1)^{k-1}
%\end{align*}
\end{proof}

%\begin{remark}
%We have $\sgn(Q_\tau) = \begin{cases} -\sgn(a) & \text{if $\tau\in\text{Int}(S_A)$} \\ \sgn(a) & \text{if $\tau\notin\text{Int}(S_A)$}\end{cases}$
%\end{remark}

%
%
%\begin{proof}[Proof of Theorem~\ref{local}]
%{\bf ToDo: Complete proof}
%The fact that $\int_{c_{A}}H_{k,k-1}(z,\tau)A(z,1)^{k-1}dz$ is bounded at $\infty$ follows from the fact for each $z_{0}$ there is a neighbourhood $V$ around $z_{0}$ that $H_{k,k-1}(z,\tau)$ is bounded at $\infty$ uniformly for $z \in V$.
%Lemmas \ref{diff} and \ref{jumps} imply that $\mathcal{F}_{1-k,A}(\tau)$ minus the two Eichler integrals is locally a constant. The sum in $P_{0,A}(\tau)$ can be read off from the singularities of $\mathcal{F}_{0,A}(\tau)$. That $c_{1}(A) = 0$ follows from the fact that $\lim_{\tau \to i\infty}\int_{c_{A}}H_{1}(z,\tau)dz = 0$, which follows from the fact that $H_{1}(z,\tau)$ vanishes at $\infty$ (as a function of $\tau$).
%\end{proof}

\begin{proof}[Proof of Theorem~\ref{theorem locally harmonic maass form identity}] By what we have said above, Theorem~\ref{theorem locally harmonic maass form identity} for $\tau$ not lying on $S_{Q}$ for any $Q \in [A]$ follows from the above lemma. By a similiar idea as in the proof of the lemma above we find that for $\tau$ lying on a semi-circle $S_{Q}$ for $Q \in [A]$ we have
	\begin{align*}
	\mathcal{C}(H_{k,k-1}(\cdot,\tau),A)
	&= \lim_{\varepsilon \to 0}\frac{1}{2}\big(\mathcal{C}(H_{k,k-1}(z,\tau+i\varepsilon),A)+\mathcal{C}(H_{k,k-1}(z,\tau-i\varepsilon),A) \big),
	\end{align*}
	where the cycle integral integral on the left-hand side is defined as the Cauchy principal value \eqref{eq Cauchy principal value}. This implies that Theorem~\ref{theorem locally harmonic maass form identity} is also true for $\tau$ lying on $S_{Q}$ for some $Q \in [A]$.
	\end{proof}

\section{The Proof of Theorem \ref{theorem traces rationality}}\label{section proof theorem traces rationality}

By Corollary~\ref{corollary main identity} we have the identity
\begin{align}\label{eq main identity}
\mathcal{C}\left(f_{k,P},A\right) = \frac{2^{k}|d|^{\frac{k-1}{2}}D^{k-\frac12}}{(k-1)! \,|\overline{\Gamma}_{P}|}R_{2-2k}^{k-1}\left(\mathcal{F}_{1-k,A}\right)(\tau_{P}).
\end{align}
Let $\mathcal{Q}$ be a finite family of indefinite quadratic forms $A \in \mathcal{Q}_{D}$ of non-square discriminants $D_{A} > 0$, and let $a_{A} \in \Z$ for $A \in \mathcal{Q}$ such that $\sum_{A \in \mathcal{Q}}a_{A}f_{k,A} = 0$. If we multiply \eqref{eq main identity} by $a_{A}$ and sum over $A \in \mathcal{Q}$, and then plug in the splitting of $\mathcal{F}_{1-k,A}$ from Theorem~\ref{local}, we see that the Eichler integrals $f_{k,A}^{*}$ and $\mathcal{E}_{f_{k,A}}$ cancel out due to the assumption $\sum_{A \in \mathcal{Q}}a_{A}f_{k,A} = 0$. Hence we obtain
\begin{align*}
\sum_{A \in \mathcal{Q}}a_{A}\mathcal{C}\left(f_{k,P},A\right) = \frac{2^{k}|d|^{\frac{k-1}{2}}}{(k-1)! \,|\overline{\Gamma}_{P}|}\sum_{A \in \mathcal{Q}}a_{A }D_{A}^{k-\frac12}R_{2-2k}^{k-1}\left(P_{1-k,A}\right)(\tau_{P}),
\end{align*}
where $P_{1-k,A}$ is the local polynomial defined in Theorem~\ref{local}. The action of the iterated raising operator on $P_{1-k,A}$ has been computed in Lemmas 5.3 and 5.4 of \cite{anbs}, and is given as follows.

\begin{lemma}\label{raisingP}
For $\tau \in \H \setminus \bigcup_{Q \in [A]}S_{Q}$ we have
\begin{align*}
 R_{2-2k}^{k-1}\left(P_{1-k,A}\right)(\tau) = \frac{(k-1)!}{2^{k}D^{k-\frac{1}{2}}}\mathcal{P}_{k,A}(\tau),
% =(-1)^{k}(k-1)!\Biggl(\frac{(\zeta_{N,A}(k)+(-1)^k \zeta_{N,-A}(k))}{2^{2k-2}(2k-1)v^{k-1} }\\ -\left(\frac{i}{2}\right)^{k-1}D^{-\frac{k}{2}}\sum_{\substack{Q = [a,b,c] \in [A] \\\tau \in \Int(C_{Q})}}\operatorname{sgn}(a) P_{k-1}\left(\frac{iQ_\tau}{\sqrt{D}}\right)\Biggr).
\end{align*}
where $\mathcal{P}_{k,A}(\tau)$ is the function defined in \eqref{eq def local polynomial}.
\end{lemma}

We arrive at
\begin{align*}
\sum_{A \in \mathcal{Q}}a_{A}\mathcal{C}\left(f_{k,P},A\right) = \frac{|d|^{\frac{k-1}{2}}}{|\overline{\Gamma}_{P}|}\sum_{A \in \mathcal{Q}}a_{A}\mathcal{P}_{k,A}(\tau_{P}),
\end{align*}
which is the formula from Theorem~\ref{theorem traces rationality}. Finally, we show that the right-hand side is rational.

\begin{lemma}\label{lemma P rational}
For any CM-point $\tau_P\in\H$ of discriminant $d < 0$, we have
\[
|d|^{\frac{k-1}{2}}\mathcal{P}_{k,A}(\tau_P)\in\Q.
\]
\end{lemma}

\begin{proof}
If $P = [a,b,c]$ with $a > 0$, then $\tau_{P}$ is given by
\[
\tau_{P} = \frac{-b+i\sqrt{|d|}}{2a}.
\]
In particular, $\frac{\sqrt{|d|}}{\Im(\tau_P)}$ and $\sqrt{|d|}Q_{\tau_{P}}$ are rational. We have seen in Proposition~\ref{proposition zeta level lowering} that 
\[
D^{k-\frac12}\left(\zeta_{\Gamma,A}(k)+(-1)^k \zeta_{\Gamma,-A}(k)\right)
\]
is a rational number for $k \geq 2$ (and this expression does not occur in $\mathcal{P}_{k,A}$ for $k = 1$). Moreover, the Legendre poynomial $P_{k-1}$ is odd if $k$ is even and even if $k$ is odd. Hence
\[
|d|^{\frac{k-1}{2}}(i\sqrt{D})^{k-1}P_{k-1}\left(\frac{iQ_{\tau_P}}{\sqrt{D}}\right)
\]
is rational. Combining all these facts we see that $|d|^{\frac{k-1}{2}}\mathcal{P}_{k,A}(\tau_P)$ is a rational number.
\end{proof}

\section{The Proof of Theorem \ref{theorem traces rationality 2}}\label{section proof theorem traces rationality 2}

For $(m,N) = 1$ the $m$-th Hecke operator $T_{m}$ on a function $f$ transforming like a modular form of weight $2k$ for $\Gamma$ is defined by
\begin{align}
\label{eq definition Tn}
f| T_{m} := m^{k-1}\sum_{M\in\Gamma\backslash \mathcal{M}_m(N)}f|_{2k}M,
\end{align}
where $\mathcal{M}_m (N)$ is the set of integral $2$ by $2$ matrices of determinant $m$ whose lower left entry is divisible by $N$, and the slash operator is defined by $(f|_{2k}M)(z) := \det(M)^{k}j(M,z)^{-2k}f(Mz)$. It acts on the Fourier expansion of a cusp form $f(z) = \sum_{n=1}^{\infty}c_{f}(n)q^{n} \in S_{2k}(\Gamma)$ by
\[
(f|T_{m})(z) = \sum_{n=1}^{\infty}\sum_{d\mid (m,n)}d^{2k-1}c_{f}(mn/d^{2})q^{n}.
\]

In order to show Theorem~\ref{theorem traces rationality 2} we would like to use the splitting of $\mathcal{F}_{1-k,A}$ from Theorem~\ref{local} and get rid of the Eichler integrals by taking suitable linear combinations. To this end, the following well-known lemma is useful.

\begin{lemma}\label{cusprelation}
	If $\underline{\lambda}$ is a relation for $S_{2k}(\Gamma)$, then $f|T_{\underline{\lambda}} = 0$ for every $f \in S_{2k}(\Gamma)$.
\end{lemma}

\begin{proof}
	We have
	\[
	(f|T_{\underline{\lambda}})(z) = \sum_{n=1}^{\infty}\lambda_{n}\sum_{m=1}^{\infty}\sum_{d \mid (m,n)}d^{2k-1}c_{f}(mn/d^{2})q^{m} = \sum_{m=1}^{\infty}\sum_{n=1}^{\infty}\lambda_{n}\sum_{d \mid (m,n) }d^{2k-1}c_{f}(mn/d^{2})q^{m}.
	\]
%	Here we used that the $m$-th coefficient of $f|T_{n}$ is the $n$-th coefficient of $f|T_{m}$. 
	Since the innermost sum is just the $n$-th coefficient of the cusp form $f|T_{m}$, the sum over $n$ vanishes by the definition of a relation for $S_{2k}(\Gamma)$.
\end{proof}

An important ingredient in the proof of Theorem~\ref{theorem traces rationality 2} is the fact that Petersson's Poincar\'e series $H_{k}(z,\tau)$ behaves nicely under the action of Hecke operators.

\begin{lemma}\label{HkHecke}
For $k \geq 1$ and $(m,N) = 1$ we have 
\[
H_k(z,\tau)|_ {z} T_m =  m^{k}  H_k(z,\tau)|_ {\tau} T_m.
\]
\end{lemma}

\begin{proof}
	For $k \geq 2$ we plug in the definition of $H_{k}(z,\tau)$ and $T_{m}$ and write
	\[
	H_{k}(z,\tau)|_{z}T_{m} = m^{k-1}\sum_{M \in \mathcal{M}_{m}(N)}\left(\frac{(z-\tau)(z-\overline{\tau})}{v} \right)^{-k}\Bigg|_{2k,z}M.
	\]
	Now a short calculation gives
	\[
	\left(\frac{(z-\tau)(z-\overline{\tau})}{v} \right)^{-k}\Bigg|_{2k,z}M = \left(\frac{(z-\tau)(z-\overline{\tau})}{v} \right)^{-k}\Bigg|_{0,\tau}M^{'},
	\]
	where $\left(\begin{smallmatrix}a & b \\ c & d \end{smallmatrix} \right)^{'} = \left(\begin{smallmatrix}d & -b \\ -c & a \end{smallmatrix} \right)$. Since $M^{'}$ also runs through $\mathcal{M}_{m}(N)$ we obtain the stated identity for $k\geq 2$. For $k = 1$ and $\Re(s) > 0$ we compute analogously
	\[
	H_{1,s}(z,\tau)|_{z}T_{m} = m H_{1,s}(z,\tau)|_{\tau}T_{m}.
	\]
	Using the well-known fact that 
	\[
	E_{2,\Gamma}^{*}(z)|_{z}T_{m} = \sigma(m)E_{2,\Gamma}^{*}(z) = |\Gamma \backslash \mathcal{M}_{m}(N)| E_{2,\Gamma}^{*}(z) = mE_{2,\Gamma}^{*}(z)|_{\tau}T_{m}
	\]
	and analytic continuation we also obtain the result for $k = 1$.
% involving the identity
%$$
%\sum_{M\in \mathcal{M}_m (N)}\left(\frac{(z-\tau)(z-\overline{\tau})}{\im(\tau)}\right)^{-k}\Biggl|_{2k, z} M = \sum_{M\in \mathcal{M}_m (N)}\left(\frac{(z-\tau)(z-\overline{\tau})}{\im(\tau)}\right)^{-k}\Biggl|_{0,\tau} M^{-1}\\
%$$
\end{proof}

We now come to the proof of Theorem~\ref{theorem traces rationality 2}. Using Lemmas~\ref{lemma fkP and Hk} and \ref{HkHecke} we compute 
\begin{align*}
\mathcal{C}(f_{k,P}|T_{m},A) &= \frac{2^{k-1}|d|^{\frac{k-1}{2}}}{|\overline{\Gamma}_{P}|\pi}\mathcal{C}(H_{k}(\cdot,\tau_{P})|_{z}T_{m},A)  = \frac{2^{k-1}|d|^{\frac{k-1}{2}}}{|\overline{\Gamma}_{P}|\pi}m^{k}\big(\mathcal{C}(H_{k}(\cdot,\tau),A)|_{\tau}T_{m}\big)(\tau_{P}) .
\end{align*}
By \eqref{eq Hk and Hkk-1} and Theorem \ref{theorem locally harmonic maass form identity} we obtain
\begin{align*}
\mathcal{C}(H_{k}(\cdot,\tau),A)|_{\tau}T_{m} &= \frac{1}{(k-1)!}\big(R_{2-2k,\tau}^{k-1}\big(\mathcal{C}(H_{k,k-1}(\cdot,\tau),A)\big)\big)|_{\tau}T_{m} \\
&=\frac{2\pi D^{k-\frac12}}{(k-1)!} \big(R_{2-2k}^{k-1}(\mathcal{F}_{1-k,A})\big) | T_m \\
&=\frac{2\pi D^{k-\frac12}}{(k-1)!}m^{k-1}R_{2-2k}^{k-1}\left(\mathcal{F}_{1-k,A} | T_m \right).
\end{align*}

%\begin{align*}
%\mathcal{C}(f_{k,P}|T_{m},A) &= \frac{2^{k-1}|d|^{\frac{k-1}{2}}}{(k-1)!|\overline{\Gamma}_{P}|\pi}m^{k}\big(\big(R_{2-2k,\tau}^{k-1}(\mathcal{C}(H_{k,k-1}(\cdot,\tau),A)\big)|_{\tau}T_{m}\big)|_{\tau = \tau_{P}} \\
%&=\frac{(-2)^k |d|^{\frac{k-1}{2}}D^{k-\frac12}}{(k-1)!|\overline{\Gamma}_{P}|}m^k \left(\left(R_{2-2k,\tau}^{k-1}\mathcal{F}_{1-k,A}\right) |_{0,\tau} T_m \right)(\tau)|_{\tau = \tau_P}\\
%&=\frac{(-2)^k |d|^{\frac{k-1}{2}}D^{k-\frac12}}{(k-1)!|\overline{\Gamma}_{P}|}m^{3k-2}\left(R_{2-2k,\tau}^{k-1}\left(\mathcal{F}_{1-k,A} |_{2-2k,\tau} T_m \right)\right)(\tau)|_{\tau = \tau_P}.
%\end{align*}

Since every coset in $\Gamma\backslash\mathcal{M}_m(N)$ is represented by a matrix $M$ with $Mi\infty =i\infty$, we have for any $f\in S_{2k}(\Gamma)$
\[
\mathcal{E}_{f}|T_m = m^{1-2k}\mathcal{E}_{f| T_m} \qquad\text{and}\qquad f^{*}| T_m =  m^{1-2k} (f|T_m)^*.
\]
This implies
\[
\mathcal{F}_{1-k,A}| T_m =  P_{1-k,A}| T_m+ m^{1-2k}(-1)^{k}\frac{D^{\frac{1}{2}-k}}{\binom{2k-2}{k-1}}\left(f_{k,A}| T_m\right)^*-m^{1-2k}(-1)^{k}D^{\frac{1}{2}-k}\frac{(k-1)!^2}{(4\pi)^{2k-1}}\mathcal{E}_{f_{k,A}| T_m}.
\]
It follows from Lemma \ref{cusprelation} that $\sum_{m>0}\lambda_{m}f_{k,A}| T_m =0$, and therefore
\begin{align*}
\sum_{m>0}\lambda_{m} \mathcal{C}(f_{k,P}|T_{m},A)
%&=\frac{(-2)^k |d|^{\frac{k-1}{2}}D^{k-\frac12}}{(k-1)!|\overline{\Gamma}_{P}|}\sum_{m>0}\lambda(-m)m^{3k-2}\left(R_{2-2k,\tau}^{k-1}\left(\mathcal{F}_{1-k,A} |_{2-2k,\tau} T_m \right)\right)(\tau)|_{\tau = \tau_P}\\
&=\frac{2^k |d|^{\frac{k-1}{2}}D^{k-\frac12}}{(k-1)!|\overline{\Gamma}_{P}|} \sum_{m>0}\lambda_{m}m^{2k-1}\left(R_{2-2k}^{k-1}\left(P_{1-k,A}| T_m\right)\right)(\tau_{P})\\
&=\frac{2^k |d|^{\frac{k-1}{2}}D^{k-\frac12}}{(k-1)!|\overline{\Gamma}_{P}|} \sum_{m>0}\lambda_{m}m^{k}\left(R_{2-2k}^{k-1}(P_{1-k,A})|T_m\right)(\tau_{P}).
\end{align*}
The expression $R_{2-2k}^{k-1}(P_{1-k,A})$ can be rewritten using Lemma \ref{raisingP}. We plug in the definition of $T_{m}$ and choose as a system of representatives for $\Gamma \backslash \mathcal{M}_{m}(N)$ the matrices $\left(\begin{smallmatrix} \alpha & \beta \\ 0 & \delta \end{smallmatrix} \right)$ with $\alpha,\beta,\delta \in \Z, \alpha > 0, \alpha\delta = m$, and $\beta \pmod \delta$. This yields the formula in Theorem~\ref{theorem traces rationality 2}. Note that $\frac{\alpha \tau_{P}+\beta}{\delta}$ is a CM point of discriminant $\delta^{2}d$. Hence Lemma~\ref{lemma P rational} implies that the expression
\[
|d|^{\frac{k-1}{2}}\mathcal{P}_{k,A}\left(\frac{\alpha \tau_{P}+\beta}{\delta} \right)
\]
is rational. This finishes the proof of Theorem~\ref{theorem traces rationality 2}.

\section{The Proof of Theorem \ref{weight 2}}\label{section proof weight 2}

Throughout this section we assume that $N$ is odd and square-free. Furthermore, we let $\Delta$ be a discriminant with $(-1)^{k}\Delta > 0$ and $\delta$ a fundamental discriminant with $(-1)^{k}\delta < 0$ such that $\delta$ is a square modulo $4N$. Finally, let $F(\tau) = \sum_{m \gg -\infty}c_{F}(m)q^{m}$ be a weakly holomorphic modular form of weight $\frac{3}{2}-k$ for $\Gamma_{0}(4N)$ in the Kohnen plus space with rational coefficients $c_{F}(m)$ for $m < 0$. We first show that the Fourier coefficients of the meromorphic modular form \eqref{eq twisted sum} are algebraic multiples of $\pi^{k-1}$.

\begin{proposition}\label{RatCoef} For $k \geq 1$ the meromorphic modular form
\[
\pi^{1-k}|\delta|^{\frac12-k}\sum_{(-1)^k\Delta > 0} c_{F}(-|\Delta|)f_{k,\Delta,\delta}
\]
has rational Fourier coefficients. 
\end{proposition}

For the proof we write the coefficients of $f_{k,\Delta,\delta}$ as linear combinations of coefficients of half-integral weight Maass Poincar\'e series.

\begin{lemma}\label{fPcoef}
Let $k \geq 1$. For $n \geq 1$ we have
\begin{multline*}
	c_{f_{k,\Delta, \delta}}(n) = -\frac{(-1)^{\left[\frac{k}{2}\right]}2^{k}\pi^{k-1}|\delta|^{k-\frac12}n^{2k-1}}{(k-1)!}\sum_{m|n} \left(\frac{\delta}{m}\right) m^{-k} c^+_{\mathcal{P}_{\frac{3}{2}-k,-|\Delta|}}\left(\frac{n^{2}|\delta|}{m^{2}}\right)\\
	\qquad+\delta_{k=1}12\prod_{p|N}\left(1-p^{-2}\right)^{-1}\sum_{d|N}\frac{\mu(d)}{d^2}\sigma\left(\frac{dn}{N}\right)\sum_{P \in \mathcal{Q}_{\Delta \delta}/\Gamma}\frac{\chi_{\delta}(P)}{|\overline{\Gamma}_P|}.
\end{multline*}
\end{lemma}

\begin{proof}
	This identity follows from a straightforward calculation using Proposition~\ref{prop fkP Fourier expansion}, Theorem~\ref{theorem Poincare Fourier expansion} and Proposition~\ref{prop salie sum}. It could alternatively be derived from the fact that $f_{k,\Delta,\delta}$ is a theta lift of $\mathcal{P}_{\frac{3}{2}-k,-|\Delta|}$, compare \cite{bringmannkanevonpippich, zemel}.
\end{proof}

\begin{proof}[Proof of Proposition~\ref{RatCoef}]
%Using Euler products and the generating series for $N=1$
%$$
%\pi^2\sum_{c\geq 1}\frac{K(n,0,c)}{c^2} = \frac{6\sigma(n)}{n},
%$$
% one can show that 
%$$
%\pi^2\sum_{c\geq 1\atop N|c}\frac{K(n,0,c)}{c^2}
%$$
%is rational for every $n$ and $N$, so we only consider the Poincar\'e series part.\\
% 

Looking at the formula for $c_{f_{k,\Delta, \delta}}(n)$ in Lemma \ref{fPcoef}, we see that the second summand on the right-hand side is rational if $k=\delta=1$ and vanishes otherwise. It remains to show that the coefficients of
\begin{align}\label{Pschmodda}
\sum_{(-1)^k \Delta > 0} c_{F}(-|\Delta|)\sum_{n \geq 1}n^{2k-1}\sum_{m|n} \left(\frac{\delta}{m}\right) m^{-k} c^+_{\mathcal{P}_{\frac{3}{2}-k,-|\Delta|}}\left(\frac{n^{2}|\delta|}{m^{2}}\right)q^n
\end{align}
are rational.

Let $F$ be a weakly holomorphic modular form of weight $\frac{3}{2}-k$. Then so is the function
	\[
	\widetilde{F}(\tau) := \sum_{m < 0}c_{F}(m)\mathcal{P}_{\frac{3}{2}-k, m}(\tau).
	\]
 If $k > 1$, we have $F = \widetilde{F}$ since there are no holomorphic modular forms of negative weight. In particular, since the space of weakly holomorphic modular forms of weight $\frac{3}{2}-k$ has a basis consisting of forms with rational coefficients and the principal part of $F$ is rational, we find that all coefficients of $F$ are rational for $k > 1$. However, for $k = 1$ the functions $F$ and $\widetilde{F}$ may differ by a holomorphic modular form. Note that every $\mathcal{P}_{\frac{3}{2}-k,m}$ is orthogonal to cusp forms with respect to the regularized Petersson inner product and has rational principal part. Hence the same is true for $\widetilde{F}$. It now follows from Proposition~3.2 in \cite{bruinierschwagenscheidt} that all Fourier coefficients of $\widetilde{F}$ are rational.  Now we see that \eqref{Pschmodda} equals
	 %For $k > 1$, using Lemma \ref{fPcoef} 	
	\[
	\sum_{n \geq 1}n^{2k-1}\sum_{m|n} \left(\frac{\delta}{m}\right) m^{-k} c_{\widetilde{F}}\left(\frac{n^{2}|\delta|}{m^{2}}\right)q^n,
	\]
which has rational Fourier coefficients. This finishes the proof.	
		%For $k = 1$ we can proceed similarly, but we have to add the coefficients of the Eisenstein series in the second line of Lemma \ref{fPcoef}.	
		%By what we have said above the coefficients of $\widetilde{F}$ are rational, which yields the claim.
\end{proof}

We now proceed to the proof of Theorem~\ref{weight 2}. For the rest of this section we let $k = 1$ and $\delta > 0$ a fundamental discriminant which is a square modulo $4N$. We can assume without loss of generality that the coefficients $c_{F}(\Delta)$ for $\Delta < 0$ are integers. We consider the differential
\[
\eta_{\delta}(F) := \pi i \sum_{\Delta < 0}c_{F}(\Delta)f_{1,\Delta,\delta}(z)dz
\]
on $X_{0}(N)$. For $P \in \mathcal{Q}_{\Delta\delta}$ we have
\[
\Res_{z = \tau_{P}}(f_{1,\Delta,\delta}(z)) = \frac{\chi_{\delta}(P)}{\pi i},
\]
so $\eta_{\delta}(F)$ has simple poles with integral residues. In particular, $\eta_{\delta}(F)$ is a differential of the third kind on $X_{0}(N)$.

Following \cite{bruinieronoheegner}, we define the twisted Heegner divisor
\[
Z_{\delta}(F) := \sum_{\Delta < 0}c_{F}(\Delta)Z_{\delta}(\Delta), \qquad Z_{\delta}(\Delta) := \sum_{P \in \mathcal{Q}_{\delta\Delta}/\Gamma}\frac{\chi_\delta(P)}{|\overline{\Gamma}_{P}|}[\tau_{P}],
\]
associated to $F$, and the corresponding degree $0$ divisor
\[
y_{\delta}(F) := Z_{\delta}(F)-\deg(Z_{\delta}(F))\cdot[i\infty].
\]
By \cite{bruinieronoheegner}, Lemma 5.1, $y_{\delta}(F)$ is defined over $\Q(\sqrt{\delta})$. Note that $y_{\delta}(F)$ is precisely the residue divisor of $\eta_{\delta}(F)$ on $X_{0}(N)$. Moreover, we have the following result.

\begin{lemma}
	The differential $\eta_{\delta}(F)$ is the canonical differential of the third kind for $y_{\delta}(F)$, i.e., the unique differential of the third kind with residue divisor $y_{\delta}(F)$ such that
	\[
	\Re\left(\int_{\gamma}\eta_{\delta}(F)\right) = 0
	\]
	for all cycles $\gamma \in H_{1}(X_{0}(N)\setminus y_{\delta}(F),\Z)$.
\end{lemma}

\begin{proof}
	One can see from Theorem~\ref{local} that $\mathcal{F}_{0,A}(\tau)\in\R$ for all $\tau\in\H$ not lying on any of the semi-circles $S_{Q}$ for $Q \in [A]$. It follows from Corollary~\ref{corollary main identity} that
\[
\Re\left(\int_{\gamma}\eta_{\delta}(F)\right) = 0
\]
if $\gamma$ is any cycle of the form $c_A$ which does not meet any poles of $\eta_{\delta}(F)$. It is well-known that the group $H_{1}(X_{0}(N)\setminus y_{\delta}(F),\Z)$ is generated by these cycles, which yields the result.
\end{proof}

The crucial ingredient for the proof of Theorem~\ref{weight 2} is the following rationality result of Scholl \cite{scholl} for differentials of the third kind (see also Theorem 3.3 of \cite{bruinieronoheegner}).

\begin{theorem}[Scholl]
	Let $D$ be a divisor of degree $0$ on $X_{0}(N)$ defined over a number field $F$. Let $\eta_{D}$ be the canonical differential of the third kind associated to $D$ and write $\eta_{D} = 2\pi i f dz$. If all the Fourier coefficients of $f$ are contained in $F$, then some non-zero multiple of $D$ is a principal divisor.
\end{theorem}

It follows from Theorem \ref{RatCoef} that the Fourier coefficients of $\frac{1}{2\pi i}\eta_{\delta}(F)$ are contained in $\Q(\sqrt{\delta})$, which is also the field of definition of the divisor $y_{\delta}(F)$. In particular, the above criterion of Scholl implies that some non-zero multiple of $y_{\delta}(F)$, say $m\cdot y_{\delta}(F)$ for some $m \in \Z$, is the divisor of a meromorphic function $g$ on $X_{0}(N)$.  

Fix some point $z_{0} \in \H$ which is not a pole of $\eta_{\delta}(F)$. For $z \in \H$ not being a pole of $\eta_{\delta}(F)$ we consider the function
\[
\Psi_{\delta}(F,z) := \exp\left(m\int_{z_{0}}^{z}\eta_{\delta}(F)\right),
\]
where the integral is over any path from $z_{0}$ to $z$ in $\H$ avoiding the poles of $\eta_{\delta}(F)$. Since the residues of $\eta_{\delta}(F)$ are integers, this does not depend on the choice of the path. Note that $\Psi_\delta(F,z)$ is meromorphic on $\H$ and has the same divisor as $g$. Thus their quotient is constant on $\H$ and $\Psi_\delta(F,z)$ is $\Gamma$-invariant.
% We remark that $\Psi_{\delta}(F,z)$ is a power of the twisted Borcherds product associated to $F$, compare \cite{bruinieronoheegner}  

For any $M\in\Gamma$, we have
\[
\Psi_{\delta}(F,Mz) = \exp\left(m\int_{z_{0}}^{Mz}\eta_{\delta}(F)\right) = \exp\left(m\int_{z_{0}}^{Mz_0}\eta_{\delta}(F)\right)\Psi_{\delta}(F,z),
\]
so for $\Psi_{\delta}(F,z)$ to be $\Gamma$-invariant, the integral
\[
\frac{m}{2\pi i}\int_{z_{0}}^{Mz_0}\eta_{\delta}(F) = \frac{m}{2}\int_{z_0} ^{Mz_0} \sum_{\Delta < 0}c_{F}(\Delta)f_{1,\Delta,\delta}(z)dz
\] 
has to be an integer. If we choose $z_0$ to lie on a geodesic $S_A$ and $M$ to be a generator of $\Gamma_A$, then this implies that the cycle integral of $\sum_{\Delta < 0}c_{F}(\Delta)f_{1,\Delta,\delta}$ along $c_{A}$ is a rational number. This finishes the proof of Theorem~\ref{weight 2}.

%\bibliography{references}{}
%\bibliographystyle{plain}

\end{document}